\documentclass[a4paper,12pt]{amsart}
\usepackage{amsmath,amsthm,amssymb,latexsym,epsfig}
\usepackage{fullpage}
\author[T. Hmidi]{Taoufik Hmidi}
\address{IRMAR, Universit\'e de Rennes 1\\ Campus de
Beaulieu\\ 35~042 Rennes cedex\\ France}
\email{thmidi@univ-rennes1.fr}
\author[J. Mateu]{Joan Mateu}
\address{Departament de Matem\`{a}tiques\\
Universitat Aut\`{o}noma de Barcelona\\
08193 Bellaterra, Barcelona, Catalonia} \email{ mateu@mat.uab.cat}
\author[J. Verdera]{Joan Verdera}
\address{Departament de Matem\`{a}tiques\\
Universitat Aut\`{o}noma de Barcelona\\
08193 Bellaterra, Barcelona, Catalonia} \email{ jvm@mat.uab.cat}

\usepackage{xcolor}
\newtheorem{teor}{Theorem}
\newtheorem*{co}{Corollary}
\newtheorem{lemma}{Lemma}

\newtheorem*{teo}{Crandall-Rabinowitz's Theorem}
\newtheorem*{sublemma}{Sublemma}

\theoremstyle{definition}

\newtheorem*{remark}{Remark}
\newtheorem*{gracies}{Acknowledgements}

\newcommand{\C}{\mathbb{C}}
\newcommand{\R}{\mathbb{R}}
\newcommand{\T}{\mathbb{T}}

\begin{document}
\title[Boundary regularity]
{Boundary regularity of rotating vortex patches} \maketitle
\begin{abstract}
We show that the boundary of a rotating vortex patch (or V-state, in
the terminology of Deem and Zabusky) is $C^{\infty}$ provided the
patch is close to the bifurcation circle in the Lipschitz norm. The
rotating patch is also convex if it is close to the bifurcation
circle in the $C^2$ norm. Our proof is based on Burbea's approach to
$V$-states.
\end{abstract}
\tableofcontents

\section{Introduction}
 The motion of a two dimensional inviscid incompressible fluid is governed by Euler equations
\begin{equation}\label{Euler}
\left\{ \begin{array}{ll}
\partial_{t} v(z,t) +(v\cdot \nabla  v)(z,t)=-\nabla p(z,t),\quad z\in\mathbb{C},\quad  t > 0,\\
\operatorname{div}v=0,\\
v(z,0)=v_{0}(z),
\end{array} \right.
\end{equation}
where $v(z,t)$ is the velocity field at the point $(z,t) \in
\C\times\R_+$ and $p$ is the pressure, which is a scalar function.
The operators $v \cdot \nabla$ and $\operatorname{div}$ are defined
by
$$
v\cdot\nabla = v_1 \partial_1+ v_2 \partial_2\quad\hbox{and}\quad
\operatorname{div} v =
\partial_1 v_1 +\partial_2 v_2.
$$
The velocity field $v$ is  divergence free because the fluid is
incompressible. In dimension two the vorticity is given by the
scalar $\omega= \partial_1 v_2 -
\partial_2 v_1.$ One can recover the velocity from the vorticity by means of the Biot-Savart law.
Indeed, identifying $v=(v_1,v_2)$ with $v_1+iv_2$  and performing a
simple calculation one gets
$$
2\partial v=i\omega,\quad\hbox{with}
\quad\partial:=\partial_z=\frac12(\partial_1-i\partial_2).
$$ Since $z\mapsto \frac{1}{\pi \, \overline{z}}$
is the fundamental solution of the complex operator $\partial$ we
get the Biot-Savart law
\begin{equation}\label{Biot}
v(z,t)=\frac{i}{2\pi} \int_\C
\frac{\omega(\zeta,t)}{\overline{z}-\overline{\zeta}}\,dA
(\zeta),\quad z \in \C,
\end{equation}
$dA$ being planar Lebesgue measure. Taking $\operatorname{curl}$ in
the first equation of the system \eqref{Euler} one obtains the
vorticity equation

\begin{equation}\label{vorticity}
\left\{ \begin{array}{ll}
\partial_{t}\omega +v\cdot \nabla\, \omega=0,\\
\omega(z,0) =\omega_{0}(z),
\end{array} \right.
\end{equation}
where  $\omega_0$ denotes the initial vorticity and $v$ is given by
the Biot-Savart law \eqref{Biot}. Equation \eqref{vorticity} simply
means that the vorticity is constant along particle trajectories.
Under mild smoothness assumptions the Euler system  is equivalent to
the vorticity formulation \eqref{vorticity}-\eqref{Biot}.
 A convenient reference for these results is
\cite[Chapter 2]{BM}.

 It is a deep fact, known as
Yudovich Theorem, that the vorticity equation has a unique global
solution in the weak sense when the initial vorticity $\omega_0$
lies in $L^1 \cap L^\infty$. See for instance \cite[Chapter 8]{BM}).
A vortex patch is the solution of \eqref{vorticity} with initial
condition the characteristic function of a bounded domain $D_0.$
Since the  vorticity is transported along trajectories, we get
\mbox{that   $\omega(z,t)={\chi}_{D_t}(z),$} where $D_t=X(D_0,t)$ is
the image  of $D_0$ by the flow. Recall that the flow $X$ is the
solution of the nonlinear integral equation
$$
X(z,t)=z+\int_{0}^t v(X(z, \tau),\tau)\, d\tau,\quad z\in \C, \,
t\geq 0.
$$
In the special case where $D_0$ is the open unit disc the vorticity
is radial and thus we get a steady flow. In particular, $D_t =D_0,
\;  t\geq0,$ and  the particle trajectories are circles centered at
the origin.  A remarkable fact discovered by Kirchoff is that when
the initial condition is the characteristic function of an ellipse
centered at the origin, then the domain $D_t$ is a rotation of
$D_0.$ Indeed, $D_t = e^{i t \Omega}\,D_0$, where the angular
velocity $\Omega$ is determined by the semi-axis $a$ and $b$ of the
initial ellipse through the formula $\Omega = ab/(a+b)^2.$ See for
instance \cite[p.304]{BM} and \cite [p.232]{L}.

A rotating vortex patch or V-state is a domain $D_0$ such that if
$\chi_{D_0}$ is the initial condition of the vorticity equation,
then the region of vorticity $1$ rotates with constant angular
velocity around its center of mass, which we assume to be the
origin. In other words, $D_t = e^{i t \Omega}\,D_0$ or, which is the
same, the vorticity at time $t$ is given by
\begin{equation*}\label{Vstate}
\omega(z,t) = \chi_{D_0}(e^{-i t \Omega} z), \quad z \in \C, \quad t
> 0.
\end{equation*}
Here the angular velocity $\Omega$ is a positive number associated
with $D_0.$

To the best of our knowledge the ellipse is the only V-state for
which a closed formula is known. Deem and Zabusky \cite{DZ} wrote an
equation for the V-states and solved it numerically. They found
V-states with $m$-fold symmetry for each integer $m \geq 2$. A
domain is said to be $m$-fold symmetric if is invariant under a
rotation of angle $2 \pi/m.$ One may view such a domain as being the
union of $m$ leaves each of which can be obtained from a given one
by rotating it of an angle of the form $p (2 \pi/m) ,$ for some
integer $p.$ It is extremely interesting to look at the pictures of
$m$-fold symmetric V-states in  \cite{WOZ} : one can see domains
with smooth boundary, which evolve with certain associated
parameters to produce in the limit domains whose boundaries have
corners at right angles.

Burbea \cite{B} gave a mathematical proof of the existence of
$m$-fold symmetric V-states using bifurcation from the circle
solution . His approach is, in Aref's words, elegant and deep
\cite[p. 346]{A}. It is indeed so. He finds an equation for the
V-states and then uses conformal mapping to rewrite the equation in
a functional analytic framework in which bifurcation theory can be
applied. Unfortunately the proof has a gap, which occurs when the
space to which bifurcation theory should be applied is set up. The
suggestion made in \cite[p.8]{B} of using the standard Hardy space
does not work. One reason is that the operator $Q(f)$ in
\cite[p.8]{B} involves one derivative of $f$ and functions in the
Hardy space do not necessarily have derivatives. Another reason is
that one needs a space which guarantees that small analytic
perturbations of the identity in the space are conformal. A space
which fulfils the preceding requirements is the space of Lipschitz
functions. However, for technical reasons, the space of Lipschitz
functions is not suitable for our purposes and has to be replaced by
the smaller space of functions with first order derivatives
satisfying a H\"{o}lder condition of order $\alpha$, $0< \alpha <1.$
The reader will find an exposition of Burbea's approach in section 2
and a complete proof of the existence of $m$- fold symmetric
V-states in section 3. It is our impression that this beautiful and
striking theorem deserves to be more widely known than it appears to
be now. In section 4 we prove our main result, which states that if
a bifurcated $V$-state is close enough to the circle in the
Lipschitz sense then its boundary is of class $C^{\infty}.$ We also
show that a bifurcated V-state is convex if it is close enough to
the circle in the $C^2$ norm. There is a dark zone, where the
V-state has boundary of class $C^{1+\alpha}$ for some positive
$\alpha$ but it is not close enough to the circle, in which the
actual smoothness properties of the boundary are unknown. Also the
nature of the singularities of the ``limiting" V-states of Wu,
Overman and Zabusky (\cite{WOZ}) are not well understood at all.

 We adhere to the convention of denoting by $C$ a constant
independent of the relevant parameters under consideration. The
constant may change its actual value at different occurrences.

{\bf{Added on October 11, 2014.}}   While working on the manuscript \cite{HHMV} we realized that 
some of the proofs presented in this paper can be substantially simplified. For instance, subsection 3.3 on
real Fourier coefficients has been now reduced to just a remark (see page 11).  Moreover, most of the technical work
which has to be done when checking that Crandall-Rabinowitz's Theorem can be applied is simplified
if one works with the differentiated form of Burbea's equation, which is (8) in page 5 of \cite{HHMV}.  See also
 \eqref{eqq}.  We have kept here the original proofs.  The alternative simpler arguments can be found in \cite{HHMV}. The reader should be aware
 that in \cite{HHMV} one works with domains with two boundary components and so the presentation has to be adapted to the easier
 simply connected case.

\section{Burbea's approach to V-states}
We first derive an equation for simply connected vortex patches
which have smooth boundary for all times. Consider two
parametrizations of the boundary $\partial D_t$ of the patch at time
$t$, say $z(\alpha,t)$ and $\eta(\beta,t)$, and assume that they are
proper in the sense that they establish a homeomorphism between the
interval of definition of the parameters with the extremes
identified and $\partial D_t.$ Assume also that they are
continuously differentiable as functions of the parameter and time.
Then there exists a change of parameters $\alpha(\beta,t)$ such that
$\eta(\beta,t) = z(\alpha(\beta,t),t)$ for all $\beta$ and $t$ and
so we have
\begin{equation*}\label{para}
\frac{\partial \eta}{\partial t} (\beta,t)=\frac{\partial
z}{\partial \alpha} (\alpha,t) \frac{\partial\alpha
(\beta,t)}{\partial t} +\frac{\partial z}{\partial t} (\alpha,t).
\end{equation*}
Since $\frac{\partial z}{\partial\alpha}(\alpha,t)$ is a tangent
vector to the boundary at the point $z(\alpha,t)$ and
$\frac{\partial \alpha}{\partial t}(\beta,t)$ is a scalar we
conclude that
\begin{equation}\label{boundary}
\frac{\partial \eta}{\partial t} (\beta,t)\cdot \vec{n}
=\frac{\partial z}{\partial t} (\alpha,t) \cdot \vec{n}
\end{equation}
where $\vec n$ is the exterior unit normal vector at the point
$z(\alpha,t)=\eta(\beta,t)$ and the dot stands for scalar product in
$\R^2 = \C.$ Thus the quantity in \eqref{boundary} does not depend
on the parametrization  and represents the speed of the
boundary in the normal direction. On the other hand
$v(z(\alpha,t),t)\cdot \vec n$ is the normal component of the
velocity of a particle which is located at the point $z(\alpha,t)$ at time
$t.$ Since the boundary is advected by the velocity $v$ we get the
equation
\begin{equation}\label{vortex}
\frac{\partial z}{\partial t} (\alpha,t)\cdot \vec{n} =v
(z(\alpha,t),t) \cdot \vec{n}
\end{equation}
which describes the motion of the boundary of the patch.

Now we introduce the stream function
\begin{equation}\label{stream}
\psi(z,t)=\left ( \frac{1}{2\pi} \log |\cdot| *
\chi_{D_{t}}\right)(z).
\end{equation}
Clearly $\overline{\partial} \log |z|^2 = 1/ \overline{z},$ where
$\overline{\partial} = \partial/ \partial \overline{z}.$  Hence
\begin{equation*}\label{distream}
2i \overline{\partial}\psi (z,t) =\frac{i}{2\pi} \left(
\frac{1}{\overline{z}}  *\chi_{D_{t}}\right) (z)= v(z,t)
\end{equation*}
and
\begin{eqnarray*}\label{dstream}
\nonumber(v \cdot \vec{n})(z(\alpha,t),t)& =& (\nabla^\perp \psi\cdot \vec{n})(z(\alpha,t),t)\\
\nonumber& =&-(\nabla
\psi \cdot \vec{\tau})(z(\alpha,t),t)\\
&=&-\frac{d\psi}{ds}(z(\alpha,t),t).
\end{eqnarray*}
where $\vec \tau$ is the unit tangent vector and $s$ is the
arc-length parameter of the curve $\partial D_t$. Therefore
\eqref{vortex} becomes
\begin{equation}\label{vortex2}
\frac{d \psi}{ds}(z(\alpha,t))= - \frac{\partial
z(\alpha,t)}{\partial t} \cdot \vec n .
\end{equation}
Assume now that the patch rotates with angular velocity $\Omega$.
Let $z_0(\alpha)$ be a proper continuously differentiable
parametrization of $\partial D_0$ and set $z(\alpha,t)= e^{i \Omega
t} z_0(\alpha),$ which is a proper parametrization of $\partial
D_t.$ Then
$$
\frac{\partial z}{\partial t}(\alpha,t)= i \Omega \,z(\alpha,t)
$$
and \eqref{vortex2} becomes
\begin{equation}\label{rvortex}
\frac{d\psi} {d s} (z(\alpha,t),t)=-i\Omega z(\alpha,t)\cdot
\vec{n}=\Omega z(\alpha,t)\cdot \vec{\tau}.
\end{equation}
Taking $\alpha=s,$ the arc-length parameter on $\partial D_0,$ we
obtain
\begin{equation*}\label{ds}
\frac{d}{ds} |z(s,t)|^{2}=2\operatorname{Re} (\overline{z(s,t)}
\vec{\tau})=2z(s,t)\cdot \vec{\tau}
\end{equation*}
so that, by \eqref{rvortex},
\begin{equation*}\label{rvortex2}
\frac{d\psi}{ds} (z(s,t),t)=\frac{\Omega}{2} \frac{d}{ds}
|z(s,t)|^{2}
\end{equation*}
and integrating with respect to $s$ yields, for a certain constant
$c(t)$ depending on $t,$
\begin{equation}\label{rvortex3}
\psi (z,t)=\frac{\Omega}{2} |z|^{2}+c(t),\quad z\in\partial D_{t}.
\end{equation}
Since the steps can be reversed this is the equation of V-states.

The goal now is to use conformal mappings to ``parametrize"
V-states. For that one needs to modify the preceding equation to get
a form more amenable to the use of analytic functions. Fix $t$ and
take derivatives in \eqref{rvortex3} with respect to $s$ on
$\partial D_t.$  We just get a restatement of \eqref{rvortex}, that
is,
\begin{equation}\label{ds2}
2\operatorname{Re} \left( \frac{\partial\psi}{\partial z} (z(s,t),t)
z'(s,t)\right)= \operatorname{Re} \left(\Omega
\overline{z(s,t)}z'(s,t)\right)
\end{equation}
where the prime means derivative with respect to $s.$ By the
generalized Cauchy formula (which follows from a direct application
of Green-Stokes) one has
\begin{equation}\label{gcauchy}
\overline{z} =\frac{1}{2\pi i} \int_{\partial
D_{t}}\frac{\overline{\zeta}}{\zeta-z} \,d\zeta+\frac{1}{\pi}
\int_{D_{t}} \frac{dA(\zeta)}{z-\zeta},\quad z\in D_{t}.
\end{equation}
Taking the $\partial = \partial / \partial z $ derivative in
\eqref{stream} and applying \eqref{gcauchy}
\begin{equation*}\label{dpsi}
4\frac{\partial\psi}{\partial z} (z,t)=\frac{1}{\pi} \int_{D_{t}}
\frac{dA(\zeta)}{z-\zeta} =\overline{z}-\frac{1}{2\pi i}
\int_{\partial D_{t}} \frac{\overline{\zeta}}{\zeta-z} \,d\zeta,
\quad z\in D_{t}.
\end{equation*}
The first identity above implies that $z\mapsto\frac{\partial
\psi}{\partial z} (z,t)$ extends continuously to the closed
\mbox{domain $\overline{D_t}$,} since the Cauchy integral of a
bounded compactly supported function is quasi Lipschitz (its modulus
of continuity is $ O(\delta |\log \delta|), \;\delta < 1/2$). Thus
the same happens to the Cauchy integral of the function
$\overline{\zeta}$ on $\partial D_t.$ Hence \eqref{ds2} becomes,
with $\lambda = 1-2 \Omega,$
\begin{equation*}\label{rvortex4}
\operatorname{Re} \left(\lambda \overline{z}\, z' (s,t)
-\frac{1}{2\pi i} \int_{\partial
D_{t}}\frac{\overline{\zeta}}{\zeta-z} \,d\zeta \,
z'(s,t)\right)=0,\quad z\in\partial D_{t},
\end{equation*}
where the integral over $\partial D_{t}$ for $z \in \partial D_{t}$
should be understood as the limit as $w \in D_{t}$ tends to $z$ of
the corresponding integral for $w.$

 Integrating with respect to $s$ on $\partial D_t$
we conclude that, for some constant $c(t)$ depending on $t,$
\begin{equation}\label{rvortex005}
\lambda |z|^{2} +2 \operatorname{Re} \frac{1}{2\pi i} \int_{\partial
D_{t}} \overline{\zeta}\log \left(1-\frac{z}{\zeta}\right)
\,d\zeta=c(t),\quad z\in\partial D_{t}.
\end{equation}
Remember that the origin belongs to $D_t$ and thus, for each $\zeta
\in \partial D_t,$ there exists a branch of the logarithm of
$1-z/\zeta$ in $D_t \setminus \{\zeta \}$ taking the value $0$ at
$z=0.$  Also fixed $z \in
\partial D_t$ there exists a branch of the logarithm of $1-z/\zeta$
in $\C_{\infty} \setminus {D_t}\cup\{z\}$ taking the value $0$ at
$\zeta=\infty.$ Therefore the integral in \eqref{rvortex005} exists
for each $z \in \partial D_t$ and defines a continuous function on
$\partial D_t.$ Later on we will differentiate this integral with
respect to $z$ on $\partial D_t.$

Notice that the equation  \eqref{rvortex005}  is invariant by
rotations and dilations. Consequently, using the fact
$D_t=e^{it\Omega}D_0$ and performing a change of variables equation
\eqref{rvortex005}  reduces to
\begin{equation}\label{rvortex5}
\lambda |z|^{2} +2 \operatorname{Re} \frac{1}{2\pi i} \int_{\partial
D_{0}} \overline{\zeta}\log \left(1-\frac{z}{\zeta}\right)
\,d\zeta=c,\quad z\in\partial D_{0}.
\end{equation}
We have used the fact $c(t)$ does not depend on $t$ since the left
hand side of \eqref{rvortex5} is independent on the time variable.
Equation \eqref{rvortex5} is Burbea's equation for simply connected
V-states.

We introduce now conformal mappings. Let $E_{D_0} = \C_{\infty}
\setminus \overline{D_0}$ be the exterior of $D_0$ and
$E_{\triangle} = \C_{\infty} \setminus \overline{\triangle}$ the
exterior of the unit disc $\triangle = \{z : |z|< 1\}.$  Let $\Phi$
be a conformal mapping of $E_\Delta$ onto $E_{D_0}$ preserving the
point at $\infty.$  This mapping can be expanded in $E_\Delta$ as
$$
\Phi(z) = a \Big(z+ \sum_{n\geq0}\frac{a_n}{z^n}\Big)
$$
for some complex number $a.$ Making a rotation in the independent
variable $z$ we can assume that $a$ is a positive number and then
dilating the domain we can further assume that $a=1.$ Following
Burbea we change variables in \eqref{rvortex5} setting $z=\Phi(w)$
and $\zeta=\Phi(\tau).$ Then \eqref{rvortex5} becomes
\begin{equation}\label{circle}
\lambda |\Phi(w)|^{2}+2 \operatorname{Re} \frac{1}{2\pi i}
\int_{\mathbb{T}} \overline{\Phi(\tau)} \log
\left(1-\frac{\Phi(w)}{\Phi(\tau)}\right) \Phi'
(\tau)\,d\tau=c,\quad w\in\mathbb{T}.
\end{equation}
Recall that we are assuming $\partial D_0$ to have rectifiable
boundary. It is then a classical result that $\Phi$ can be extended
continuously to the unit circle $\T = \{z \in \C : |z|=1\}$ and the
extension is absolutely continuous, so that $\Phi'(w)$ exists for
almost all $w \in \T$ and the resulting function is in $L^1(\T)$
(see \cite{P}).  Later on we will work with a $\Phi$ whose extension
to $\T$ is of class $C^{1+\alpha}(\T)$ so that $\Phi'$ will be
H\"{o}lder continuous of order $\alpha$ on $\T.$ To simplify the
notation, set
\begin{equation}\label{essa}
\sigma(\Phi) (w) =\frac{1}{2\pi i} \int_{\mathbb{T}}
\overline{\Phi(\tau)} \log
\left(1-\frac{\Phi(w)}{\Phi(\tau)}\right)\Phi' (\tau)\,d\tau,\quad
w\in\mathbb{T}
\end{equation}
and
\begin{equation}\label{mitjana}
m(\Phi,\lambda)=\frac{1}{2\pi} \int_{\mathbb{T}} (\lambda |\Phi
(w)|^{2}+2\operatorname{Re} \sigma(\Phi)(w))\,|dw|,
\end{equation}
where $|dw|$ denotes the length measure on the circle. Hence
\eqref{circle} is
\begin{equation}\label{circle1}
\lambda |\Phi(w)|^{2} +2\operatorname{Re} \sigma(\Phi)(w) -
m(\Phi,\lambda)=0,\quad w\in\mathbb{T}.
\end{equation}
It is more convenient to set $\Phi(z)=z + f(z)$ with $f$ analytic on
$E_\Delta$ and define the operator
\begin{equation}\label{ops1}
S(f)(w)=\sigma(\Phi) (w).
\end{equation}
That this is a good point of view is confirmed by the fact that
\begin{equation}\label{zero}
S(0) (w)=\frac{1}{2\pi i} \int_{\mathbb{T}} \log
\left(1-\frac{w}{\tau}\right) \frac{d\tau}{\tau}=0,\quad
w\in\mathbb{T},
\end{equation}
because the integrand is analytic on $E_\triangle$ and has a double
zero at $\infty.$ Thus \eqref{circle1} is satisfied for $\Phi(z)=z$
and each $\lambda.$ Now define
\begin{equation}\label{circle001}
F(\lambda,f)(w):=\lambda |w+f(w)|^{2} +2\operatorname{Re} {S}(f)(w)
- m(\hbox{I}+f,\lambda),\quad w\in\mathbb{T},
\end{equation}
where $I$ stands for the identity function. Clearly Burbea's
equation can be rewritten as
$$
F(\lambda,f)=0.
$$
One has
$$
 F(\lambda,0)= 0,\quad \lambda \in \R,
$$
which simply says that the disc satisfies Burbea's equation
\eqref{rvortex5} for each $\lambda$. Burbea's idea at this point is
to apply bifurcation theory in order to prove the existence of
$m$-fold V-states. We will apply in the next section
Crandall-Rabinowitz's Theorem whose original statement in
\cite[p.325]{CR} is included below for the reader's convenience. For
a linear mapping $L$ we let $N(L)$ and $R(L)$ stand for the kernel
and range of $L$ respectively. If $Y$ is a vector space and $R$ is a
subspace, then $Y/R$ denotes the quotient space.
\begin{teo}\label{CR}
Let $X$, $Y$ be Banach spaces, $V$ a neighborhood of $0$ in $X$ and
$$
F\colon (-1,1) \times V\to Y
$$
have the properties
\begin{enumerate}
\item[(a)] $F(t,0)=0$ for any $|t|<1$.
\item[(b)] The partial derivatives $F_{t}$, $F_{x}$ and $F_{tx}$ exist and are continuous.
\item[(c)] $N(F_{x}(0,0))$ and $Y/R(F_{x}(0,0))$ are one-dimensional.
\item[(d)] $F_{tx}(0,0) x_{0}\notin R(F_{x}(0,0))$, where
$$
N(F_{x}(0,0))=\operatorname{span}\{x_{0}\}.
$$
\end{enumerate}
If  $Z$ is any complement of $N(F_{x}(0,0))$ in $X$, then there is a
neighborhood~$U$ of $(0,0)$ in $\mathbf{R} \times X$, an
interval~$(-a,a)$, and continuous functions $\varphi\colon (-a,a)\to
\mathbf{R}$, $\psi\colon (-a,a)\to Z$ such that  $\varphi(0)=0$,
$\psi(0)=0$ and
$$
F^{-1} (0)\cap U= \{(\varphi(\xi), \xi x_{0}+\xi \psi(\xi)):| \xi|
<a\}\cup \{(t,0): (t,0)\in U\}.
$$
\end{teo}

\section{Existence of $m$-fold V-states}
In this section we will apply Crandall-Rabinowitz's Theorem to prove
the existence of $m$-fold V-states for each integer $m \geq 2.$ For
$m=2$ one recovers the Kirchoff ellipses. As a by-product of this
formalism we get  a H\"{o}lderian boundary
 regularity result for the V-states which are close to a point of the bifurcation set $\{(\lambda,f)=(\frac1m,0), m = 2,3, \dots \}.$
 Later we will see how to establish the   $C^\infty$ regularity of the boundary of these V-states by using  hidden smoothing effects of the nonlinear
  \mbox{equation \eqref{circle1}.} First, we establish the following result.

\begin{teor}
Given $0 < \alpha < 1$ and $m=2,3,...$  there exists a curve of
$m$-fold rotating vortex patches with boundary  of class
$C^{1+\alpha}$ bifurcating at the circle solution.

More precisely, there exist $a
> 0$ and continuous functions  $\lambda \colon (-a,a)\to
\mathbf{R}$, $\psi\colon (-a,a)\to C^{1+\alpha}(\T)$ satisfying
$\lambda(0)= 1/m$, $\psi(0)=0$, such that the Fourier series of
$\psi(\xi)$ is of the form
$$
\psi(\xi)(w)= a_{2m-1}(\xi) \overline{w}^{2m-1}+ \dots + a_{n
m-1}(\xi) \overline{w}^{n m-1}+\dots, \quad w \in \T,
$$
and
$$
F(\lambda(\xi), \xi \overline{w}^{m-1}+ \xi \psi(\xi)(w))=0, \quad
\xi \in (-a,a).
$$
The mapping
$$
\Phi_\xi(z) = z (1+ \xi \frac{1}{z^{m}} + \xi a_{2m-1}(\xi)\,
\frac{1}{z^{2m}} + \dots + \xi a_{n m-1}(\xi)\, \frac{1}{z^{n
m}}+\dots )
$$
is conformal and of class $C^{1+\alpha} $on $\C \setminus \Delta,$
and the complement $D_\xi$ of $\Phi_\xi (\C \setminus \Delta)$ is an
$m$-fold  rotating vortex patch.
\end{teor}

The proof of this theorem requires some lengthly work, which will be
presented in several steps. We start by introducing  the spaces $X$
and $Y$ and then we will check all the assumptions of
Crandall-Rabinowitz's Theorem.

\subsection{The spaces $X$ and $Y$}
The choice of the spaces $X$ and $Y$ is a key point, which was
overlooked in \cite{B}. Before giving the complete description of
these spaces we need first to recall the definition of the
H\"{o}lder spaces $C^{n+\alpha}(\Omega).$ Let $\Omega$ be a nonempty
open set of $\mathbb{R}^d$ and  $0 < \alpha <1.$ We denote by
$C^\alpha(\Omega)$ the space of continuous functions $f$ such that
$$
\|f\|_{C^\alpha(\Omega)}:=\|f\|_{L^\infty}+\sup_{x\neq
y\in\Omega}\frac{|f(x)-f(y)|}{|x-y|^\alpha}<\infty,
$$
where $\|f\|_{L^\infty}$ stands for the supremum norm of $f$ on
$\Omega$.
More generally, for a non-negative integer $n$ the H\"{o}lder space
$C^{n+\alpha}(\Omega)$ consists of  those functions of class $C^n$
whose $n-$th order derivatives are  H\"{o}lder continuous with
exponent $\alpha$ in $\Omega.$ It is equipped with the norm
$$
\|f\|_{C^{n+\alpha}(\Omega)}=\|f\|_{L^\infty}+\sum_{|\gamma|=n
}\|\partial^\gamma f\|_{C^\alpha(\Omega)}.
$$
We will also make use of the space $C^{1+\alpha}(\T)$ which is the
set of continuously differentiable functions $f$ on the unit circle
$\T$ whose derivatives satisfy a H\"{o}lder condition of order
$\alpha,$ endowed with the norm
$$
\|f\|_{C^{1+\alpha}(\T)}=\|f\|_{L^\infty}+\Big\|\frac{df}{dw}\Big\|_{\alpha},
$$
where $ \|\cdot\|_{\alpha} $ is the usual Lipschitz semi-norm of
order $\alpha$
\begin{equation*}
\|g\|_\alpha = \sup_{x\neq y\in\T} \frac{|g(x)-g(y)|}
{|x-y|^\alpha}\cdot
\end{equation*}
We define in a similar way the spaces $C^{n+\alpha}(\T)$, for each
positive integer $n$ and $\alpha\in ]0,1[.$ A word on the operator
$d/dw$ is in order. Any function  $f:\mathbb{T}\to\R$  can be
identified with a $2\pi-$periodic function $g:\R\to \R$ via the
formula
$$
f(w)=g(\theta),\quad w=e^{i\theta}.
$$
Therefore for a smooth function $f$ we get
$$
\frac{df}{dw}=-i e^{-i\theta} g^\prime(\theta).
$$
 It
will be more convenient in the sequel to work with $d/dw$ instead of
$d/d\theta.$  Since they differ only by a smooth factor it really
makes no difference. Notice that we have the identity
\begin{equation}\label{diff1}\frac{d \{\overline{f}\}}{dw} =
-\frac{1}{{w}^2} \overline{\frac{df}{dw}}.
\end{equation}
 On the other hand, if we denote by $C^{1+\alpha}_{2\pi}(\R)$ the subspace of
 $C^{1+\alpha}(\R)$ consisting of
 $2\pi$- periodic functions,  then we can identify  $C^{1+\alpha}(\mathbb{T})$ with $C^{1+\alpha}_{2\pi}(\R)$ and
$$
\|g\|_\alpha\approx \|f\|_{\alpha}.
$$
Let $C^{1+\alpha}_a(\Delta^c)$ be the space of analytic functions on
$\C_\infty \setminus \overline{\triangle}$ whose derivatives satisfy
a H\"{o}lder condition of order $\alpha$ up to $\T.$ This is also
the space of functions in $C^{1+\alpha}(\T)$  whose Fourier
coefficients of positive frequency vanish.

The space $X$ is defined as
\begin{equation}\label{expansion}
X=\Big\{f\in C^{1+\alpha}(\T); f(w) = \sum_{n=0}^{\infty} a_n
\overline{w}^n, \,w\in\mathbb{T},\;a_n\in\R, n \geq 0
\Big\}\end{equation} and coincides with the subspace of
$C^{1+\alpha}_a(\Delta^c)$ consisting of those functions in
$C^{1+\alpha}_a(\Delta^c)$ whose boundary values have real Fourier
coefficients.
 Later on we will modify
appropriately the space $X$ to get $m$-fold symmetry, but for now it
is simpler and clearer to work with this $X.$ The reader will
understand later why we require the $a_n$ to be real. For the time
being we just comment on the geometric meaning of this requirement.
If we set, for $f \in X,$ $\Phi(z)= z+f(z)$ and $\Phi$ happens to be
conformal on $\C_\infty \setminus \overline{\triangle}$, then the
complement in $\C_\infty$ of the closure of $ \Phi(\C_\infty
\setminus \overline{\triangle})$ is a simply connected domain $D$
symmetric with respect to the real axis. Conversely, if one starts
with a bounded simply connected domain $D$ symmetric with respect to
the real axis and $\Phi$ is the conformal mapping of the complement
of the closed unit disc onto the complement of $\overline{D},$ then
the coefficients in the expansion of $\Phi$ as a power series in
$1/z$ are real. Now, if one is given a domain $D$ with an axis of
symmetry containing the origin, after a rotation one can assume that
this axis is the real line. Domains with $m$-fold symmetry have an
axis of symmetry, so it is not really restrictive for our purposes
to work with functions with real Fourier coefficients.

Let $V$ stand for the open ball with center $0$ and radius $1$ in
$C^{1+\alpha}_a(\Delta^c)$. If $f$ is in $V$ then the function
$\Phi(z)=z+f(z)$ is analytic on $\{z : |z| > 1)$  and is injective
there. For, by the maximum principle,
$$
 \Big\|\frac{df}{dw}\Big\|_{L^\infty(\Delta^c)} = \sup \Bigg\{\frac{|f(z)-f(w)|} {|z-w|} :
|z|\geq1, |w|\geq1, z \neq w \Bigg\} : = \delta < 1,
$$
and hence
$$
|\Phi(z)-\Phi(w)| \geq |z-w|-|f(z)-f(w)| \geq (1-\delta)|z-w|, \quad
|z| \geq 1, |w| \geq 1.
$$
For $f \in V$ define $S(f)$ as in \eqref{ops1} and \eqref{essa},
where $\Phi(z)=z+f(z).$ The definition makes sense precisely because
$\Phi$ is injective and thus a branch of the logarithm of
$1-\Phi(w)/\Phi(\tau)$ can be defined taking the value $1$ at
$\infty$ (that is,  as $\tau \rightarrow \infty$), as we argued
after \eqref{rvortex5}. We define now a function $F(\lambda,f)$ on
$\R \times V$ by
\begin{equation}\label{efa}
F(\lambda,f)(w)= \lambda |w+f(w)|^2+ 2 \,\text{Re } S(f)(w)-
m(\hbox{Id}+f,\lambda),
\end{equation}
where $S(f)$ is as in \eqref{ops1} and $m(\hbox{Id}+f,\lambda)$ as
in \eqref{mitjana}. This  is the function $F$ to which we will apply
Crandall-Rabinowitz's Theorem.

We define now the space $Y$ as the subspace of $C^{1+\alpha}(\T)$
consisting of real-valued functions with zero integral and real
Fourier coefficients. More precisely,
\begin{equation}\label{Yspace}
Y:=\Big\{g\in C^{1+\alpha}(\T) : g(w)=\sum_{0 \neq
n\in\mathbb{Z}}a_n w^n, \, w\in\T \; \text{and} \; a_n=a_{-n}\in
\mathbb{R}, n
> 0 \Big\}.
\end{equation}

 Since we have subtracted the mean in \eqref{efa} it is
clear that $F(\lambda,f)$ is a real-valued function with zero
integral. To show that $F$ maps $X$ into $Y$ it remains to show that
$F(\lambda,f)$ belongs to   $C^{1+\alpha}(\T)$ and that its Fourier
coefficients are real.

\subsection{$F(\lambda,f)$ is in $C^{1+\alpha}(\T)$}  Since $F(\lambda, f)$ has zero integral
its norm in the space $C^{1+\alpha}(\T)$ is equivalent to the
$C^{\alpha}(\T)$ norm of its derivative
$w\mapsto\frac{dF}{dw}(\lambda,f)(w).$
 Now, since $C^{1+\alpha}(\T)$ is an algebra  the problem reduces to showing  that $w\mapsto\frac{dS(f)}{dw}(w) $ belongs to $C^{\alpha}(\T).$  For that
we need to compute the derivative of $S(f)(w)$ with respect to $w$
and this is done in the next lemma. Recall that $\Phi:\T\to\C$ is
bilipschitz (into the image) if for a positive constant $C$ one has
$$C^{-1}|\tau-\omega| \le |\Phi(\tau)-\Phi(\omega)|\le C
|\tau-\omega|, \quad \tau, \omega \in \T .$$

\begin{lemma}\label{deefa} For any bilipschitz function $\Phi:\T\to\C$  of class $C^1(\T)$ we have
\begin{equation}\label{dessa}
\frac{d}{dw} S(f)(w) = -\Phi'(w) \left(\frac{\overline{\Phi(w)}}{2}+
\hbox{p.v.\ }  \frac{1}{2 \pi i} \int_{\T}
\frac{\overline{\Phi(\tau)}\Phi'(\tau) }{\Phi(\tau)-\Phi(w)}\,d\tau
\right), \quad  w\in \T.
\end{equation}
\end{lemma}

\begin{proof}
The proof of this lemma consists in computing the derivative with
respect to $w$ in the sense of distributions by integrating against
a test function on $\T.$ When integrating by parts we will get
$-\Phi'(w)$ times the principal value integral in \eqref{dessa} plus
a ``boundary term" $-\Phi'(w) \overline{\Phi(w)} / 2 $, which is due
to the fact that the logarithm in the definition of $S(f)$ is not
continuous on the diagonal.

An alternative argument goes as follows. Bring the integral defining
$S(f)$ on the $C^{1}$ Jordan curve  $\Gamma= \Phi(\T)$  to obtain
the function
$$
\sigma(z)= \frac {1}{2\pi i} \int_\Gamma \overline{\zeta} \, \log
(1-\frac{z}{\zeta})\,d\zeta, \quad z \in \Gamma.
$$
Think of $\sigma$ as an analytic function of $z \in D,$ the domain
enclosed by $\Gamma.$  Its derivative is
$$
-\frac {1}{2\pi i} \int_\Gamma
\frac{\overline{\zeta}}{\zeta-z}\,d\zeta = -\overline{z} -
\frac{1}{\pi} \int_D \frac{1}{\zeta-z}\,dA(\zeta), \quad z \in D,
$$
where the identity is the generalized Cauchy formula for the
function $\overline{z}$. The integral on $D$ in the right-hand side
is a quasi-Lipschitz function, as the Cauchy integral of a bounded
function, and hence the left-hand side extends continuously up to
the boundary. Call this extension $-C(\overline{\zeta})(z), \; z \in
\Gamma,$   where the notation refers to ``boundary values of the
Cauchy integral of the function $\overline{\zeta}$ " . By Plemelj's
formula (see, for example, \cite[p. 143]{V})
\begin{equation}\label{plemelj}
C(\overline{\zeta})(z)= \frac{\overline{z}}{2} + \text{p.v.\ }
\frac{1}{2 \pi i} \int_\Gamma \frac{\overline{\zeta}}{\zeta-z}\,
d\zeta.
\end{equation}
We conclude that $\sigma$ is differentiable on $\Gamma$ and its
derivative with respect to $z$ is given by minus the expression in
\eqref{plemelj}. Changing variables to return to the unit circle and
applying the chain rule we get \eqref{dessa}.
\end{proof}

Now everything is reduced to checking that the the principal value
integral in \eqref{dessa} satisfies a H\"{o}lder condition of order
$\alpha.$  This can be done rather easily in at least two ways. The
first consists in considering the operator
\begin{equation}\label{operator}
Tg(w) =\text{p.v.\ }  \frac{1}{2\pi i} \int_{\T}
\frac{g(\tau)}{\Phi(\tau)-\Phi(w)}\,d\tau,\quad w\in\,\T,
\end{equation}
and showing that $T$ maps boundedly $C^\alpha(\T)$ into itself. This
can be achieved by applying the $T(1)-$ Theorem for H\"{o}lder
spaces of Wittmann \cite[Theorem 2.1, p.584]{W} (see also
\cite{Ga}). Before stating Witmann's result recall that the maximal
singular integral of $g$ is
$$
T^*(g)(w)= \sup_{\epsilon > 0} \left| \int_{\T \setminus
D(w,\epsilon)}\frac{g(\tau)}{\Phi(\tau)-\Phi(w)}\,d\tau \right|
,\quad w\in\T,
$$
$D(w,\epsilon)$ being the disc centered at $w$ of radius $\epsilon.$
Wittmann's $T(1)-$ Theorem asserts in our context that the
$C^\alpha(\T)$ boundedness of $T$ follows by checking that the
kernel is ``standard", that $T^*(1)$ is bounded and that the
operator $T$ applied to the constant function $1$ lies in
$C^\alpha(\T).$ That the kernel is standard means in the situation
we are considering that
\begin{equation}\label{standard1}
\left|\frac{1}{\Phi(\tau)-\Phi(w)}\right| \le \frac{C}{|\tau-w|},
\quad \tau,w\in \mathbb{T}
\end{equation}
and
\begin{equation}\label{standard2}
\left|\frac{d}{dw} \left(\frac{1}{\Phi(\tau)-\Phi(w)}\right)\right|
\le \frac{C}{|\tau-w|^{2}}, \quad \tau,w\in \mathbb{T}
\end{equation}
which are clearly satisfied because $\Phi$ is bilipschitz on its
domain. To compute $T(1)(w), w \in \T,$ denote by
$\gamma_\epsilon,\; \epsilon>0,$ the arc which is the intersection
of the circle centered at $w$ of radius $\epsilon$ and the
complement of the open unit disc, with the counter-clockwise
orientation. Let $\T_\epsilon$ be the closed Jordan curve consisting
of the arc $\gamma_\epsilon$ followed by the part of the unit circle
at distance from $w$ not less than $\epsilon.$  We claim that
\begin{eqnarray}\label{tione}
\nonumber T(1)(w) &=&\lim_{\varepsilon\to 0} \left( \frac{1}{2\pi i}
\int_{\mathbb{T}_{\varepsilon}} \frac{\tau-w}{\Phi(\tau)-\Phi(w)}
\frac{d\tau}{\tau-w} -\frac{1}{2\pi i}\int_{\gamma_{\epsilon}}
\frac{\tau-w}{\Phi(\tau)-\Phi(w)}
\frac{d\tau}{\tau-w}\right)\\
 &=&1- \frac{1}{2\Phi'
(w)}
\end{eqnarray}
The integral over $\T_\epsilon$ is $1$ since the integrand is
analytic as a function of $\tau$ on the exterior of the unit disc
and we have
$$ \lim_{|\tau|\to\infty} \frac{\tau-w}{\Phi(\tau)-\Phi(w)}=1.
$$
The limit as $\epsilon \rightarrow 0$ of the integral over
$\gamma_\epsilon$ is
\begin{equation*}\label{gammaepsilon}
\frac{1}{\Phi'(w)} \lim_{\epsilon \rightarrow 0}
\int_{\gamma_\epsilon} \frac{d\tau}{\tau-w} = \frac{\pi
i}{\Phi'(w)},
\end{equation*}
and so \eqref{tione} is proven.
 From the assumption  $\Phi\in C^{1+\alpha}(\T)$ combined with \eqref{tione}  we obtain   that $T(1)
\in C^\alpha(\T).$  It is also clear that $T^*(1)$ is bounded on
$\T$ from the argument above.

For future reference we record now the following identity, whose
proof is similar to that just described of \eqref{tione},
\begin{equation}\label{identit123}
\text{p.v.}\frac{1}{2 \pi
i}\int_{\T}\frac{\Phi^\prime(\tau)}{\Phi(\tau)-\Phi(w)}d\tau=
\frac{1}{2}\cdot
\end{equation}

The second argument to show that the principal value integral in
\eqref{dessa} satisfies a H\"{o}lder condition of order $\alpha$
consists in changing variables $\zeta=\Phi(\tau)$ and $z=\Phi(w)$
and pass to a principal value integral on the $C^{1+\alpha}$ Jordan
curve $\Gamma=\Phi(\T).$ The kernel of the operator one obtains is
the Cauchy kernel and the integration is with respect to $d\zeta.$
This operator sends $C^{\alpha}(\Gamma)$ into itself on curves
satisfying a mild regularity assumption called Ahlfors regularity.
This can be proved by standard arguments and we prefer to omit the
lengthly calculations and inequalities involved.
\subsection{Real Fourier coefficients}
We intend to  show that if $f \in X$ then $F(\lambda,f)$ has real
Fourier coefficients. This is very simple once we notice that an integrable function $g$ on $\T$ has real Fourier coefficients if
and only if $g(\overline{w})= \overline{g(w)}.$  Examining the definition of $F(\lambda,f)$  in \eqref{efa}
we conclude that only the term $S(f)(w)$ has to be considered, the others having clearly real Fourier coefficients.  The identity $S(f)(\overline{w}) = \overline{S(f)(w)}$ follows from
the obvious change of variables $z=\overline{\tau}$.
\subsection{F is Gateaux differentiable}
We show here that $F$ is Gateaux differentiable and in the next
subsection we will check that the directional derivatives are
continuous. This will show that $F$ is continuously differentiable
on its domain $\R \times V.$ In proving differentiability properties
of $F(\lambda,f)$ we can consider only the first two terms in
\eqref{circle001} because the third, $m(I+f,\lambda),$ is a mean of
these  two terms. Denote by $G(\lambda,f)$ the sum of the first two
terms in \eqref{circle001}.

There is no problem with the partial derivative with respect to
$\lambda$ because the dependence on $\lambda$ is linear. We get
\begin{eqnarray*}\label{dlambda}
D_{\lambda} F(\lambda,f)(w)=|w+f(w)|^{2}
\end{eqnarray*}
which is obviously continuous with respect to $f$ in the topology of
$X.$\\
Take now $h\in X$ and compute the derivative of $G(\lambda,f)$ with
respect to $f$ in the direction $h,$  that is,
\begin{eqnarray}\label{partialefa}
\nonumber D_{f} G(\lambda,f)(h)&=& \frac{d}{dt} G(\lambda, f+th)_{\bigr\rvert_{t=0}}\\
&=&2\lambda \operatorname{Re} (\Phi \overline{h})
+2\operatorname{Re} \frac{d}{dt} S(f+th)_{\bigr\rvert_{t=0}}.
\end{eqnarray}
The computation of $\frac{d}{dt} S(f+th)|_{t=0}$  yields the four
terms below,
\begin{eqnarray}\label{partialessa}
\nonumber\frac{d}{dt} S(f+th)_{\bigr\rvert_{t=0}}
(w)&=&\frac{1}{2\pi i}\int_{\T} \overline{h(\tau)} \log
\left(1-\frac{\Phi(w)}{\Phi(\tau)} \right) \Phi'
(\tau)\,d\tau\\
\nonumber&+&+\frac{1}{2\pi i} \int_{\T} \overline{\Phi(\tau)} \log
\left(1-\frac{\Phi(w)}{\Phi(\tau)}
\right) h'(\tau)\,d\tau\\
\nonumber&+&\frac{1}{2\pi i} \int_{\T} \overline{\Phi(\tau)}
\frac{h(w)-h(\tau)}{\Phi(\tau)-\Phi(w)}
\Phi'(\tau)\,d\tau\\
\nonumber&+&\frac{1}{2\pi i}
\int_{\T}\frac{\overline{\Phi(\tau)}}{\Phi(\tau)}
h(\tau)\Phi'(\tau)\,d\tau\\
&:=& A(f,h) (w)+B(f,h)(w)+C(f,h)(w)+D(f,h),
\end{eqnarray}
where the last identity is the definition of the functions $A,B,C$
and $D.$ Notice that $D$ is independent of $w.$\\
We proceed now to prove that $A(f,h) \in C^{1+\alpha}(\T).$ For the
estimate of the absolute value of $A(f,h)(w)$ we can assume without
loss of generality that $w=1.$ Computing the derivative of
$\log(1-\Phi(1)/\Phi(\tau))$ with respect to $\tau$ in $|\tau|> 1$
and using the fundamental theorem of calculus we see that
\begin{equation}\label{log}
\log \left( 1-\frac{\Phi(1)}{\Phi(\tau)}\right)= -\int^{\infty}_{1}
\frac{\Phi(1)\Phi'(t\tau)}{(\Phi(t\tau)-\Phi(1))
\Phi(t\tau)}\tau\,dt,\quad 1\ne \tau,\quad |\tau|\geq1.
\end{equation}
The preceding identity is very useful in estimating the integrals
containing a logarithmic term, as we will see below. We proceed now
to estimate from below the factors in the denominator of the
fraction inside the integral in \eqref{log}. The function
$\Phi(\tau)/\tau$ is analytic on $\C \setminus \overline{\Delta}$
and takes the value $1$ at $\infty.$ Hence, by the maximum principle
and recalling that $f \in V,$ we obtain
\begin{eqnarray}\label{fisobre}
1-\|f\|_{L^\infty(\T)} \le \frac{|\Phi(\tau)|} {|\tau|} \le 1+
\|f\|_{L^\infty(\T)}, \quad |\tau|\ge 1,
\end{eqnarray}
and
\begin{equation}\label{coerciv1}
|\Phi(\tau)-\Phi(1)|\geq (1-\|f'\|_{L^\infty(\T)})|\tau-1|,\quad
|\tau|\geq 1.
\end{equation}
Therefore, by \eqref{log}, \eqref{fisobre} and \eqref{coerciv1} we
get for $\tau\in\T\backslash\{1\}$,
\begin{equation*}\label{log1}
\begin{split}
\left| \log \left(1- \frac{\Phi(1)}{\Phi(\tau)} \right) \right| &\le
|\Phi(1)| \|\Phi'\|_{L^\infty(\Delta^c)}\int^{\infty}_{1}
\frac{dt}{|\Phi (t\tau)-\Phi(1)| |\Phi(t\tau)|}\\*[7pt]
 &\le
\frac{\big(1+\|f\|_{L^\infty(\T)}\big)
\big(1+\|f^\prime\|_{L^\infty(\T)}\big)}{\big(1-\|f\|_{L^\infty(\T)}\big)
\big(1-\|f^\prime\|_{L^\infty(\T)}\big)}
\int_{1}^{\infty}\frac{dt}{t|t\tau -1|}\cdot
\end{split}
\end{equation*}
We split the interval of integration in the last integral above in
three subintervals : \\ $(1, 1+|\tau-1|), (1+|\tau-1|, 3)$ and
$(3,\infty).$ In the first we notice that for $t\geq 1$ and
$|\tau|=1$  we have the elementary inequality
$$
|t\tau-1|= |t-\overline{\tau}|=|t-\tau| \geq |1-\tau|
$$
and so
$$
\int_{1}^{1+|\tau-1|}\frac{dt}{t|t\tau -1|} \le 1.
$$
The integral on the second interval can be estimated
straightforwardly as follows
$$
\int_{1+|\tau-1|}^{3}\frac{dt}{t|t\tau - 1|} \le
\int_{1+|\tau-1|}^{3}\frac{dt}{t - 1} = \log \frac{2}{|\tau-1|}\cdot
$$
Finally
$$
\int_{3}^{\infty}\frac{dt}{t|t \tau -1|} \le
\int_{3}^{\infty}\frac{dt}{t(t-1)} = \log \frac{3}{2}.
$$
Therefore, collecting the preceding inequalities,
\begin{equation}\label{log2}
\left| \log \left(1- \frac{\Phi(1)}{\Phi(\tau)} \right) \right|  \le
C(f) \left(1+\big|\log {|\tau -1|}\big|\right), \quad \tau \in \T,
\end{equation}
where $C(f)$ is a constant depending only on the $C^1(\T)$ norm of
$f.$\\
Integration in $\tau$ on $\T$ readily yields
$$
\|A(f,h)\|_\infty \le C_1(f) \|h\|_\infty.
$$
The next step is to estimate the uniform norm and the Lipschitz
semi-norm of order $\alpha$ of the function $w\mapsto
\frac{d}{dw}A(f,h)(w).$ As in Lemma \ref{deefa} one has
\begin{equation}\label{deA}
\frac{d}{dw} A(f,h)(w) = -\Phi'(w) \left(\frac{\overline{h(w)}}{2}+
\text{p.v.\ } \frac{1}{2 \pi i} \int_{\T}
\frac{\overline{h(\tau)}\,\Phi'(\tau) }{\Phi(\tau)-\Phi(w)}\,d\tau
\right), \quad w\in\T,
\end{equation}
and the only difficulty lies in estimating the $C^\alpha(\T)$ norm
of the principal value integral. But this has already been done in
subsection 3.2. Therefore $A(f,h) \in C^{1+\alpha}(\T).$

The proof that the term $B(f,h)$ belongs to $C^{1+\alpha}(\T)$ is
basically the same. The expression for the derivative is
\begin{equation*}\label{dib}
\frac{d}{dw} B(f,h) (w)= -\Phi' (w) \left(
\frac{\overline{\Phi(w)}h'(w)}{2\Phi' (w)}+\text{p.v.\
}\frac{1}{2\pi i} \int_{\T} \frac{h'(\tau)
\overline{\Phi(\tau)}d\tau}{\Phi(\tau)-\Phi(w)}\right)
\end{equation*}
and one deals with the principal value integral as before, using the
operator in \eqref{operator}.

The proof that the term $C(f,h)$ is in $C^{1+\alpha}(\T)$ contains a
little variation of the preceding argument. First, we observe that
the quotient
$$
\frac{h(\tau)- h(w)}{\Phi(\tau)- \Phi(w)}
$$
takes continuously the value $h'(w)/\Phi'(w)$ on the diagonal of
$\T$. Consequently, in computing the derivative of $C(f,h)(w)$ no
boundary terms will arise and we get
\begin{eqnarray*}\label{dic}
\nonumber\frac{d}{dw} C(f,h) (w)&=& h' (w) \text{ p.v.\
}\frac{1}{2\pi i} \int_{|\tau|=1} \frac{g(\tau)}
{\Phi(\tau)-\Phi(w)}\,d\tau\\
 &+& \Phi' (w)
\text{ p.v.\ } \frac{1}{2\pi i} \int_{\T}
\frac{h(w)-h(\tau)}{(\Phi(\tau)-\Phi(w))^{2}}g(\tau)\,d\tau
\end{eqnarray*}
where $g(\tau)= \overline{\Phi(\tau)}\,\Phi'(\tau).$ The mapping
properties of the operator $T$ in \eqref{operator} take care of the
principal value integral in the first term. We view the principal
value integral in the second term as an operator $U$ acting on $g.$
Its kernel is standard, because $\Phi$ is bilipschitz, and the
action of $U$ on the constant function $1$ is, as in \eqref{tione},
\begin{eqnarray*}\label{uone}
\nonumber U(1)(w)&=& \lim_{\varepsilon\to 0}
\int_{\mathbb{T}_{\varepsilon}}
\frac{h(w)-h(\tau)}{(\Phi(\tau)-\Phi(w))^{2}}\,d\tau\\
\nonumber&+&\lim_{\varepsilon\to 0} \frac{1}{2\pi i}
\int_{\gamma_{\varepsilon}}
\frac{h(\tau)-h(w)}{\tau-w}\left(\frac{\tau-w}{\Phi(\tau)-\Phi(w)}\right)^{2}
\frac{d\tau}{\tau-w}\\
&=&\frac{h'(w)}{2\Phi'(w)^{2}}\cdot
\end{eqnarray*}
In the integral on $\T_\epsilon$ the integrand is analytic on $\C
\setminus \overline{\triangle}$ and has a double zero at $\infty,$
and so the integral vanishes. The limit as $\epsilon$ tends to $0$
of the integral on $\gamma_\epsilon$ is $i \pi$ times the limit as
$\tau$  tends to $w$ of the quotients inside, that is,
$h'(w)/\Phi'(w)^2.$ Then $U(1) \in C^\alpha(\T).$ The previous
discussion also gives that $U^*(1)$ is bounded.  Thus one can apply
Wittmann's $T(1)-$Theorem (see the paragraph after \eqref{operator})
and conclude that $U$ maps $C^\alpha(\T)$ into itself. This
completes the proof that $w\mapsto D_f F(\lambda,f)(h)(w) \in
C^{1+\alpha}(\T).$ On the other hand, it is clear that $D_f
F(\lambda,f)(h)$ depends linearly on $h.$ Therefore $F(\lambda,f)$
is Gateaux differentiable at any point $(\lambda,f) \in \R \times
V.$

\subsection{Continuity of $D_f F(\lambda,f)$}
In this subsection we prove that the mapping $D_f F(\lambda,f)$ is
continuous as a function of $f \in V$ taking values in the space of
bounded linear operators from $X$ into $Y.$  In particular this
shows that $F(\lambda,f)$ is continuously differentiable in the
Frechet sense. What one has to do is the following. Fix $f \in V$
and show an inequality of the type
\begin{equation*}\label{cont}
\|D_{f}F(\lambda,f)(h) - D_{g} F(\lambda,g)
(h)\|_{C^{1+\alpha}(\mathbb{T})}\le C_{1+\alpha}(f)
\|f-g\|_{C^{1+\alpha}(\mathbb{T})} \|h\|_{C^{1+\alpha}(\mathbb{T})}
\end{equation*}
for $g \in X$ close enough to $f.$ Here we denote by
$C_{1+\alpha}(f)$ a constant depending on the \mbox{norm
$\|f\|_{C^{1+\alpha}}$.} By \eqref{partialefa} and
\eqref{partialessa} this amounts to prove similar inequalities for
$A(f,h), B(f,h)$ and $C(f,h)$ in place of $D_f F(\lambda,f).$ For
instance,
\begin{equation*}\label{conta}
\|A(f,h) -A (g,h)\|_{C^{1+\alpha}(\mathbb{T})}\le C_{1+\alpha}(f)
\|f-g\|_{C^{1+\alpha}(\mathbb{T})} \|h\|_{C^{1+\alpha}(\mathbb{T})}
\end{equation*}
for $g \in X$ close to $f.$ We start with the estimate of the
uniform norm of the difference
\begin{equation}\label{contauniform}
\|A(f,h) -A (g,h)\|_{\infty}\le C_1(f) \|f-g\|_{C^1(\T)}
\|h\|_{L^\infty(\T)}.
\end{equation}
To prove the uniform estimate above it is enough to consider the
point $w=1$ in $\T.$  Take $g \in V$ close to $f$ and set
$\Psi(\tau)= \tau+g(\tau), \;|\tau| \geq 1.$ We can easily check
that
\begin{eqnarray}\label{A}
\nonumber A(f,h)(1)-A(g,h)(1)&=& \frac{1}{2 \pi i} \int_{\T}
\overline{h(\tau)} \log
\Big(1-\frac{\Phi(1)}{\Phi(\tau)}\Big)\big(\Phi'(\tau)-\Psi^\prime(\tau)\big)d\tau\\
\nonumber&+& \frac{1}{2 \pi i} \int_{\T}
\overline{h(\tau)}\Psi^\prime(\tau)\left( \log
\Big(1-\frac{\Phi(1)}{\Phi(\tau)}\Big)-\log
\Big(1-\frac{\Psi(1)}{\Psi(\tau)}\Big)\right) \,d\tau \\
&:=&\mathcal{I}_1(f,g)(h)+\mathcal{I}_2(f,g)(h).
\end{eqnarray}
To estimate the first term $\mathcal{I}_1(f,g)(h)$ we use
\eqref{log2},

\begin{eqnarray}\label{contauniform1}
\nonumber \big|\mathcal{I}_1(f,g)(h)\big| &\le&
\frac{1}{2\pi}\|h\|_{\infty}
\|f^\prime-g^\prime\|_{\infty}\int_{\T}\Big|\log
\Big(1-\frac{\Phi(1)}{\Phi(\tau)} \Big)  \Big||d\tau|\\
&\le& C(f) \|h\|_{\infty} \|f^\prime-g^\prime\|_{\infty}.
\end{eqnarray}
To treat  the second term $\mathcal{I}_2(f,g)(h)$, we  use the
identity \eqref{log} which yields
\begin{eqnarray*}
\mathcal{K}(\tau)&=&\int_{1}^{\infty}
\left(\frac{\Phi(1)\Phi'(t\tau)}{(\Phi(t\tau)-\Phi(1)) \Phi
(t\tau)}-\frac{\Psi(1)\Psi'(t\tau)}{(\Psi(t\tau)-\Psi(1)) \Psi
(t\tau)}\right)\tau\,dt  , \quad \tau\in \T\backslash\{1\},
\end{eqnarray*}
where
$$
\mathcal{K}(\tau):=\log \Big(1-\frac{\Psi(1)}{\Psi(\tau)}\Big)-\log
\Big(1-\frac{\Phi(1)}{\Phi(\tau)}\Big).
$$
By elementary algebraic computations one gets
\begin{eqnarray*}
\mathcal{K}(\tau)&=& \int_{1}^{\infty}
\frac{\big(\Phi(1)-\Psi(1)\big)\Phi'(t\tau)}{(\Phi(t\tau)-\Phi(1))
\Phi (t\tau)}\tau\,dt+\Psi(1)\int_{1}^{\infty}
\frac{\Phi'(t\tau)-\Psi^\prime(t\tau)}{(\Phi(t\tau)-\Phi(1)) \Phi
(t\tau)}\tau\,dt\\
&+&\Psi(1)\int_{1}^{\infty}\Psi^\prime(t\tau)
\frac{\Psi(t\tau)-\Phi(t\tau)}{(\Phi(t\tau)-\Phi(1)) \Phi
(t\tau)\Psi(t\tau)}\tau\,dt\\
&+&\Psi(1)\int_{1}^{\infty}\frac{\Psi^\prime(t\tau)}{\Psi(t\tau)}
\Big(\frac{1}{\Phi(t\tau)-\Phi(1)}-\frac{1}{\Psi(t\tau)-\Psi(1)}\Big) \tau\,dt\\
&:=&\sum_{j=1}^{4}\mathcal{K}_j(\tau).
\end{eqnarray*}
Coming back to the identity \eqref{log}, the  first term
$\mathcal{K}_1(\tau)$ takes the form
$$
\mathcal{K}_1(\tau)=-\frac{\Phi(1)-\Psi(1)}{\Phi(1)}\log
\Big(1-\frac{\Phi(1)}{\Phi(\tau)}\Big).
$$
It follows according to \eqref{log2} that
$$
\big|\mathcal{K}_1(\tau)\big|\le
C_1(f)\|f-g\|_{L^\infty(\T)}\Big(1+\big|\log|\tau-1|\big|\Big).
$$
As in the proof of \eqref{log2}, we obtain
\begin{eqnarray*}
\big|\mathcal{K}_2(\tau)\big|&\le& |\Psi(1)|\|\Phi^\prime-\Psi^\prime\|_{L^\infty(\Delta^c)}\int_{1}^{\infty}\frac{1}{|\Phi(t\tau)-1||\Phi(t\tau)|}d\tau\\
&\le& |\Psi(1)|\|f^\prime-g^\prime\|_{L^\infty(\T)}\int_{1}^{\infty}\frac{1}{|\Phi(t\tau)-1||\Phi(t\tau)|}d\tau\\
&\le&C_1(f)\|f^\prime-g^\prime\|_{L^\infty(\T)}\Big(1+\big|\log|\tau-1|\big|\Big).
\end{eqnarray*}
Concerning the third term, we write
\begin{eqnarray*}
\big|\mathcal{K}_3(\tau)\big|&\le& |\Psi(1)|\|\Psi^\prime\|_{L^\infty(\Delta^c)}\|\Phi-\Psi\|_{L^\infty(\Delta^c)}\int_{1}^{\infty}\frac{1}{|\Phi(t\tau)-1||\Phi(t\tau)||\Psi(t\tau)|}d\tau\\
&\le&C_1(f)\|f-g\|_{L^\infty(\T)}\Big(1+\big|\log|\tau-1|\big|\Big).
\end{eqnarray*}
To treat the last term $\mathcal{K}_4$ we use \eqref{coerciv1}
\begin{eqnarray*}
\Big|\frac{1}{\Phi(t\tau)-\Phi(1)}-\frac{1}{\Psi(t\tau)-\Psi(1)}\Big|&=&\Big|\frac{\{\Psi-\Phi\}(t\tau)-\{\Psi-\Phi\}(1)}{(\Phi(t\tau)-\Phi(1))(\Psi(t\tau)-\Psi(1))}\Big|\\
&\le& 4\frac{\|\Psi^\prime-\Phi^\prime\|_{L^\infty(\Delta^c)}|t\tau-1|}{|t\tau-1|^2}\\
&\le& 4\frac{\|f^\prime-g^\prime\|_{L^\infty(\T)}}{|t\tau-1|}.
\end{eqnarray*}
Consequently,
\begin{eqnarray*}
|\mathcal{K}_4(\tau)|\le
C_1(f)\|f^\prime-g^\prime\|_{L^\infty(\T)}\big(1+\big|\log|\tau-1|\big|\big).
\end{eqnarray*}
Therefore,
$$
|\mathcal{K}(\tau)|\le C_1(f)\|f-g\|_{C^1(\T)}
\big(1+\big|\log|\tau-1|\big|\big), \quad \tau\in\T\backslash\{1\}.
$$
Hence, coming back to the definition of $\mathcal{I}_2(f,g)(h)$ in
\eqref{A} and integrating  in $\tau$
$$
|\mathcal{I}_2(f,g)(h)|\le C_1(f)\|f-g\|_{C^1(\T)}.
$$
Finally,
\begin{equation*}\label{contauniform4}
|A(f,h) (1)-A (g,h)(1)|\le C_1(f)
\|h\|_{L^\infty(\T)}\|f-g\|_{C^1(\T)}
\end{equation*}
which is \eqref{contauniform}.

Our next task is to estimate the difference
\begin{equation*}\label{difA}
\frac{d}{dw} A(f,h)(w)-\frac{d}{dw} A(g,h)(w)
\end{equation*}
in $C^\alpha(\T)$ for $g$ close to $f$ in $C^{1+\alpha}(\T).$ The
difficult term in the formula for the derivative of the function  $
A(f,h)(w)$ in \eqref{deA} is the principal value integral
\begin{eqnarray}\label{contauniformde}
\nonumber I(f,h)(w) &=& \text{p.v.\ } \frac{1}{2\pi i}\int_{\T}
\frac{\overline{h(\tau)}\Phi'(\tau)} {\Phi
(\tau)-\Phi(w)}\,d\tau\\
 \nonumber&=&\frac{1}{2\pi i} \int_{\T} \frac{\overline{h(\tau)}-\overline{h(w)}} {\Phi (\tau)-\Phi(w)}\Phi'(\tau)\,d\tau +\frac{1}{2} \overline{h(w)}\\
 &:= &J(f,h) (w) +\frac{1}{2} \overline{h(w)}.
\end{eqnarray}
Thus the difference $I(f,h)(w) - I(g,h)(w)$ is
\begin{eqnarray*}
J(f,h)(w)-J(g,h)(w)&=& \frac{1}{2\pi i}\int_{\T}
\frac{\overline{h(\tau)}-\overline{h(w)}} {\Phi
(\tau)-\Phi(w)} (\Phi'(\tau)-\Psi'(\tau))\,d\tau\\
&+&\frac{1}{2\pi i}\int_{\T}
\big({\overline{h(\tau)}-\overline{h(w)}}\big)
\Psi^\prime(\tau)\Big(\frac{1}{{\Phi
(\tau)-\Phi(w)}}-\frac{1}{ \Psi(\tau)-\Psi(w)}\Big)\,d\tau\\
&:=&\mathcal{T}_1(f,g)(w)+\mathcal{T}_2(f,g)(w).
\end{eqnarray*}
Now we need an estimate for $\mathcal{T}_1(f,g)$ and
$\mathcal{T}_2(f,g)$ in $C^\alpha(\T).$ Both terms have the form
\begin{equation}\label{contauniformde4}
T\chi(w)= \int_{\T} K(w,\tau) \chi(\tau)\,d\tau,\qquad \chi\in
L^{\infty}(\mathbb{T})
\end{equation}
where the kernel $K(w,\tau)$ and the function $\chi$ are
\begin{equation*}\label{contauniformde5}
K(w,\tau)= \frac{\overline{h(\tau)}- \overline{h(w)}}
{\Phi(\tau)-\Phi(w)}, \qquad \chi(\tau)= \Phi' (\tau)-\Psi'(\tau)
\end{equation*}
in $\mathcal{T}_1(f,g)(w)$ and
\begin{equation*}\label{contauniformde6}
K(w,\tau)=\frac{\overline{{h}(\tau)}-\overline{h(w)}}{\Phi(\tau)-\Phi(w)}-\frac{\overline{{h}(\tau)}-\overline{h(w)}}{\Psi(\tau)-\Psi(w)},
\quad \chi(\tau)=\Psi^\prime(\tau)
\end{equation*}
in $\mathcal{T}_2(f,g)(w).$

\begin{lemma}\label{kernel}
Assume that the kernel of the operator $T$ in
\eqref{contauniformde4} satisfies
\begin{enumerate}
\item
$K$ is measurable on $\T \times \T$ and
$$|K(w,\tau)| \le C_0, \quad w, \tau \in \T.$$

\item
For each $\tau \in \T$, $w\mapsto K(w,\tau)$ is differentiable in
$\T\backslash\{w\}$ and
$$
\left |\partial_w K(w,\tau) \right| \le \frac{C_0}{|w-\tau|}, \quad
w,\tau \in \T, \; w \neq \tau.
$$
\end{enumerate}
Then
\begin{equation}\label{liplog}
|T\chi(w_1)-T\chi(w_2)| \le C\|\chi\|_{L^\infty}\,C_0 \left(1+\log
\frac{1}{|w_1 -w_2|} \right) |w_1 -w_2|
 , \quad w_1\neq w_2 \in \T,
\end{equation}
for some constant $C.$
 In particular, $T\chi  \in C^\alpha(\T)$ for every $0 < \alpha < 1.$
\end{lemma}
The proof of this lemma
 is simple and standard. For the details of a similar
result the reader is referred to \cite[p.419]{MOV}.

To get the desired estimates for  $\mathcal{T}_1(f,g)$ and
$\mathcal{T}_2(f,g)$ we just need to  check that their kernels
satisfy the hypothesis of Lemma \ref{kernel}. We deal first with
$\mathcal{T}_1(f,g).$  From \eqref{coerciv1} with the point $1 \in
\T$ replaced by an arbitrary $w \in \T$, one readily gets
$$
|K(w,\tau)|\le C\|h^\prime\|_{L^\infty(\T)},\quad \left |\partial_w
K(w,\tau) \right| \le C\|h^\prime\|_{L^\infty(\T)}|\tau-w|^{-1}.
$$
Therefore
\begin{eqnarray*}
\|\mathcal{T}_1(f,g)\|_{C^\alpha(\T)}&\le& C \,\|h^\prime\|_{L^\infty(\T)} \|\Psi^\prime-\Phi^\prime\|_{L^\infty(\T)}\\
&\le&C \,\|h^\prime\|_{L^\infty(\T)} \|f-g\|_{C^1(\T)}.
\end{eqnarray*}
To estimate the kernel of $\mathcal{T}_2(f,g)$ we use
\eqref{coerciv1}
\begin{eqnarray*}
\big|K(w,\tau)  \big|&=&\frac{|h(\tau)-h(w)|\big|(\Phi-\Psi)(w)- (\Phi-\Psi)(\tau) \big|}{|\Phi(\tau)-\Phi(w)|\,|\Psi(\tau)-\Psi(w)|}\\
&\le&C\|h^\prime\|_{L^\infty(\T)}\|\Phi^\prime-\Psi^\prime\|_{L^\infty(\T)}.
\end{eqnarray*}
The derivative of $K(w,\tau)$ with respect to $w$ can also be
estimated easily :
\begin{eqnarray*}
\big|\partial_wK(w,\tau)\big| &\le&|h^\prime(w)|\Big|\frac{1}{\Phi(\tau)-\Phi(w)}-\frac{1}{\Psi(\tau)-\Psi(w)}\Big|\\
&+&{|h(\tau)-h(w)|}\Bigg|\frac{\Phi^\prime(w)}{\big(\Phi(\tau)-\Phi(w)\big)^2}-\frac{\Psi^\prime(w)}{\big(\Psi(\tau)-\Psi(w)\big)^2}\Bigg|\\
&\le& C\|h^\prime\|_{L^\infty(\T)}\|\Phi^\prime-\Psi^\prime\|_{L^\infty(\T)}\Big(1+\|\Phi^\prime+\Psi^\prime\|_{L^\infty(\T)} \Big)\frac{1}{|\tau-w|}\\
&\le& C_1(f)\|h\|_{C^1(\T)}\|f-g\|_{C^1(\T)}\frac{1}{|\tau-w|}\cdot
\end{eqnarray*}
This gives according to Lemma \ref{kernel}
\begin{eqnarray*}
\|\mathcal{T}_2(f,g)\|_{C^\alpha(\T)}&\le&
C_1(f)\|h\|_{C^1(\T)}\|f-g\|_{C^1(\T)}.
\end{eqnarray*}
Hence we get
$$
\big\|I(f,h)-I(g,h) \big\|_{C^{1+\alpha}(\T)}\le
C_1(f)\|h\|_{C^1(\T)}\|f-g\|_{C^1(\T)}
$$
and, finally, gathering all previous estimates
\begin{eqnarray*}
\big\|A(f,h)-A(g,h) \big\|_{C^{1+\alpha}(\T)}
&\le& C_1(f)\|f-g\|_{C^{1+\alpha}(\T)}\|h\|_{C^{1+\alpha}(\T)}.
\end{eqnarray*}
This concludes the proof of the continuity of the term $A(f,h)(w)$
with respect to $f.$

The proof for the terms $B(f,h)$, $C(f,h)$ and $D(f,h)$ given by
\eqref{partialessa} follows a similar pattern. We omit the details.

\subsection{Kernel and range of $D_f F(\lambda,0)$}

In this subsection we study the kernel and the range of $D_f
F(\lambda,0).$  We also find the ``eigenvalues", that is, the values
of $\lambda$ for which the kernel of $D_f F(\lambda,0)$ is not
trivial. In fact,  the dimension of the kernel for these particular
values of $\lambda$ turns out to be $1$. Then we check that the
range has codimension $1$, so that Crandall-Rabinowitz's Theorem can
be applied.

Letting $f=0$ in \eqref{partialefa} and \eqref{partialessa} we
obtain
\begin{equation*}\label{deefazero}
D_f
F(\lambda,0)(h)(w)=2\lambda\,\operatorname{Re}(h(w)\overline{w})+2
\operatorname{Re} \Big(\frac{1}{2\pi
i}\int_{\mathbb{T}}\overline{h(\tau)}\log(1-\frac{w}{\tau})d\tau\Big)-m(h),
\end{equation*}
where $m(h)$ is the mean on $\T$ with respect to $|dw|$ of the sum
of the first two terms in the right-hand side above. Hence $w\mapsto
D_f F(\lambda,0)(h)(w)$ has zero integral with respect to $|dw|$ on
$\T.$ We would like to compute the Fourier series of $w\mapsto D_f
F(\lambda,0)(h)(w)$ in terms of the Fourier series of $h$
\begin{equation}\label{fourierh}
h(w)=\sum_{n=0}^{\infty} b_n\,\overline{w}^n,\quad w\in\T.
\end{equation}
Since $h \in X$ the Fourier coefficients $b_n$ of $h$ are real.
Using the expansion
$$\log(1- \frac{w}{\tau})= - \sum_{n=1}^\infty \frac{1}{n} w^n \overline{\tau}^n$$
and computing we get
\begin{equation*}\label{fourierDF}
D_f
F(\lambda,0)(h)(w)=\sum_{n=1}^{\infty}(\lambda-\frac1n)\,b_{n-1}\,w^n+\sum_{n=1}^{\infty}(\lambda-\frac1n)\,b_{n-1}\,\overline{w}^n.
\end{equation*}
From the above expression we immediately conclude that the kernel of
$D_f F(\lambda,0)$ is non trivial only if $\lambda= 1/m$ for some
positive integer $m.$ If this is the case, then the kernel is one
dimensional and is generated by the function $w\mapsto
\overline{w}^{m-1}= 1/ w^{m-1}.$  It is precisely at this point when
we use the fact that the Fourier coefficients of the functions in
our space $X$ are real. If the coefficient were complex, we would
get a kernel of real dimension $3.$ Let us now look at the range of
$D_f F(\lambda,0)$ under the assumption that $\lambda = 1/ m.$
Clearly
\begin{equation}\label{fourierDF1}
D_f F(1/m,0)(h)(w)=
2\sum_{n=1}^{\infty}\big(\frac1m-\frac1n\big)\,b_{n-1}\,
\cos(n\theta), \quad w=e^{i \theta}.
\end{equation}
Notice that $D_f F(1/m,0)(h)$ is a function with zero integral and
real Fourier coefficients. We have shown in subsection 3.4 that it
is in $C^{1+\alpha}(\T)$ and thus in the space  $Y$ given by
\eqref{Yspace}.  The only Fourier frequency missing in the expansion
\eqref{fourierDF1} is $m$, so that it looks plausible that a
complement of the range of $D_f F(1/m,0)$ is the one dimensional
subspace generated by $\cos (m \theta).$ To prove this we need to
show that each $g \in Y$ with an expansion of the form
\begin{equation*}\label{expansiong}
g(w)=\sum_ {n=1,\, n\neq m}^{\infty} \beta_n\,\overline{w}^n+\sum_{
n=1,\, n\neq m}^{\infty} \beta_n\,{w}^n.
\end{equation*}
with real $\beta_n$ is equal to  $D_f F(1/m,0)(h)$ for some $h \in
X.$ If $h$ is as in \eqref{fourierh} with real Fourier coefficients
then the equation $D_f F(1/m,0)(h) = g$ is equivalent to
\begin{equation*}\label{equationh}
2\big(\frac1m-\frac1n\big)b_{n-1}=\beta_n, \quad  n=1,2,..
\end{equation*}
or, solving for $b_n,$
\begin{eqnarray*}
 \nonumber b_n & =& \frac{m}{2}\,\frac{n+1}{n+1-m}\,\beta_{n+1}\\
&=&\frac{m}{2}\,\beta_{n+1}+\frac{m^2}{2}\,\frac{1}{n+1-m}\,\beta_{n+1}.
\end{eqnarray*}
The solution to $D_f F(1/m,0)(h)=g$ is
\begin{equation*}\label{solutionh}
h(w)=\frac{m}{2}\,w \,G(w)+\frac{m^2}{2}\,w\, H(w)
\end{equation*}
where
\begin{equation*}\label{G}
G(w)=\sum_{n=1}^{\infty}\beta_{n}\,\overline{w}^n,\quad w\in\T
\end{equation*}
and
\begin{equation*}\label{H}
H(w)=\sum_{ n=1, \;n\ne m}^{\infty}\frac{\beta_n}{n-m}\,
\overline{w}^n,\quad w\in\T.
\end{equation*}
The function $G$ is in $C^{1+\alpha}(\T)$, and then in $X,$ because
the Cauchy projection
\begin{equation*}\label{Cauchyproj}
\sum_{n=-\infty}^{\infty}c_n w^n\mapsto \sum_{n=0}^{\infty}c_n w^n
\end{equation*}
preserves the space $C^{1+\alpha}(\T)$. This is false for the space
$C^{1}(\T)$ of continuously differentiable functions on $\T$
(because the Cauchy projection does not preserve $L^\infty(\T)$) and
this is why we cannot choose the space $C^{1}(\T)$ in the definition
of $X$ and $Y$. It still remains to show that $H \in
C^{1+\alpha}(\T),$ but this is easy. Set
$$K(w)= \sum_{ n=1,\; n\neq m}^\infty \frac{\overline{w}^n}{n-m}, \quad w\in\T,$$
so that $K \in L^2(\T) \subset L^1(\T)$ and $H=G * K \in
C^{1+\alpha}(\T).$

To apply Crandall-Rabinowitz's Theorem we still have to check the
transversality condition, that is, the the second order partial
derivative $D_{f \lambda} F(1/m,0)$ of $F$ with respect to $f$ and
$\lambda $ applied to the function $w\mapsto \overline{w}^{m-1}$ is
not in the range of $D_f F(1/m,0).$ Now $D_{f \lambda} F(1/m,0)$ can
be identified with a bounded linear mapping from $X$ into $Y.$ It is
easy to see that
$$
D_{f \lambda} F(1/m,0)(h)(w) = 2 \text{Re} (h(w)\overline{w}), \quad
h \in X.
$$
Hence
$$
D_{f \lambda} F(1/m,0)(\overline{w}^{m-1})(w) = 2 \text{Re}
(\overline{w}^{m}) = 2 \cos(m \theta), \quad w=e^{i\theta},
$$
which is not in the range of $D_f F(1/m,0).$

Finally one checks easily that $D_{f \lambda} F(\lambda,f)(h)(w) = 2
\operatorname{Re}(\Phi(w)\overline{h(w)})$ is a continuous function
on $\R \times V.$
\subsection{$m$-fold symmetry}
We showed in the previous subsections how to apply
Crandall-Rabinowitz's Theorem to the spaces $X$ and $Y.$ The
conclusion is that, given a positive integer $m$, we have a
continuous curve $(\lambda_\xi, f_\xi) \in \R \times V$, defined for
$\xi$ in some interval of the form $(1/m -\delta, 1/m +\delta),$
such that $F(\lambda_\xi, f_\xi) = 0, \; \xi \in (1/m -\delta, 1/m
+\delta).$ Then $\Phi_\xi(z)=z+f_\xi(z), \; |z|\geq 1,$ is a
conformal mapping of $\C_\infty \setminus {\overline{\triangle}}$
into some domain $U_\xi$ and $D_\xi = \C_\infty
\setminus{\overline{U_\xi}}$ is a simply connected vortex patch
which rotates with angular velocity $\Omega_\xi =
(1-\lambda_\xi)/2.$ We know that $D_\xi$ is a domain with boundary
of class $C^{1+\alpha}$, but nothing else can be said about its
symmetry properties without further arguments . The $m$-fold
symmetry follows by adding a condition to the spaces $X$ and $Y.$\\
Given $m$, define $X_m$ as the subspace of $X$ consisting of those
functions $f \in X$ with a Fourier expansion of the type
\begin{equation*}\label{expansionm}
f(w)= \sum_{n=1}^{\infty}a_{nm-1} \overline{w}^{nm-1},\quad w\in \T.
\end{equation*}
If $f$ is in the open unit ball of $X_m$ the expansion of the
associated conformal mapping $\Phi$ \mbox{in $\{z : |z| \ge 1 \}$}
is given by
\begin{equation*}\label{expansionmfi}
\Phi(z)= z \big(1+\sum_{n=1}^{\infty} \frac{a_{nm-1}}{z^{nm}}\big).
\end{equation*}
This will provide the $m$-fold symmetry of the associated patch, via
the relation
$$\Phi(e^{i 2 \pi/m} z)= e^{i 2 \pi/m} \Phi(z), \; |z|\geq1
.$$ The space $Y_m$ is the subspace of $Y$ consisting of those $g
\in Y$ whose Fourier coefficients vanish at frequencies which are
not non-zero multiples of $m.$ In other words, the Fourier expansion
of $g$ is of the type
\begin{equation*}\label{expansiony}
g(w)= \sum_{n=0}^\infty \beta_{nm}\, 2 \cos(nm\theta), \quad
w=e^{i\theta},
\end{equation*}
with $\beta_0 = 0$ and real $\beta_{nm}.$ Notice that, since the
generator $w\mapsto \overline{w}^{m-1}$ of the kernel of $D_f
F(1/m,0)$ is in $X_m$ for $m \ge 2$, we still have that the
dimension of the kernel is $1.$ In the same way the codimension of
the range of  $D_f F(1/m,0)$ in $Y_m$ is $1.$

However to apply Crandall-Rabinowitz's Theorem to $X_m$ and $Y_m$
one has to check that $F(\lambda,f) \in Y_m$ if $f \in X_m.$ This
follows rather easily from work we have already done. One has to
observe that the space $A_m$ of continuous functions on $\T$ whose
Fourier coefficients vanish at frequencies which are not integer
multiples of $m$ is an algebra, closed in the space of continuous
functions on $\T$.  Next we remark that $w\mapsto \Phi'(w) \in A_m.$
We also need the fact that $w\mapsto \frac{w}{\Phi(w)} \in A_m.$ To
show this, set
$$g(w)= \frac{\Phi(w)}{w} - 1, \quad w \in \T,$$
so that $g \in A_m$ and $\|g\|_\infty < 1.$ Thus
$$
w\mapsto \frac{w}{\Phi(w)} = \sum_{n=0}^{\infty}(-1)^{n}g^{n}(w)\in
A_m.
$$

An easy computation gives that $w\mapsto \lambda |\Phi(w)|^2$
belongs to $A_m$ if $f \in X_m.$ It remains to show that $w\mapsto
S(f)(w)$ is in $A_m.$ Recall the identities \eqref{essa2} and
\eqref{Asubnk}. First, $a_n \neq 0$ only for indexes of the form
$n=m q-1$ for some positive integer $q.$ On the other hand, $A_{nk}$
is the Fourier coefficient corresponding to the frequency $n-k+1$ of
the function $\Phi'(\tau) (\tau / \Phi(\tau))^k,$ which is in $A_m.$
Hence $A_{nk}$ is non-zero only if $n-k+1 = m r$ for some integer
$r.$ Therefore the sum in $k$ is only over indexes which are
multiples of $m.$ It remains to examine at the Fourier coefficients
of $\Phi(w)^k.$ Now $\Phi(w)= w g(w)$ with $g \in A_m$ and so
$\Phi(w)^k = w^k g(w)^k$ is also in $A_m$ because only indexes $k$
which are multiples of $m$ have to be taken into account.

Therefore we can apply Crandall-Rabinowitz's Theorem to $X_m$ and
$Y_m$ and finally obtain the existence of $m$-fold symmetric
V-states for each integer $m \ge 2.$

\subsection{Kirchhoff's ellipses}
For $m=2$ we obtain the ellipses parametrized by $w\in \T\mapsto
w+\xi \overline{w}.$  The real number $\xi$ satisfies $-1 < \xi < 1$
and is a parameter which determines the shape of the ellipse. The
ellipse is centered at $0$, has horizontal semi-axis $1+\xi, \; $
and vertical semi-axis $1-\xi.$ The function $z\mapsto
z+\frac{\xi}{z} $ is the conformal mapping
 of the exterior on the unit disc onto the exterior of
the ellipse.  It is instructive to use Crandall-Rabinowitz's Theorem
to prove that these ellipses rotate. We are going to apply the
Theorem to the one dimensional spaces $X,$ which is generated by
$\overline{w},$ and $Y,$ generated by $w^2 +\overline{w}^2 = 2
\cos(2\theta), \; w=e^{i \theta}.$ Notice that $X$ is the kernel of
$D_f F(1/2,0),$ which then has range $\{0\}$ of codimension $1$ in
$Y.$ Of course we have to check that $F(\lambda,f)$ sends $X$ into
$Y.$ Take $f(w)= \xi \overline{w}$ with $|\xi| < 1,$ so that
$\Phi(w) = w + \xi \overline{w}.$  The term
$$
|\Phi(w)|^2 = 1+ \xi^2 + \xi w^2 + \xi \overline{w}^2
$$
is correct because the constant $1+\xi^2$ will disappear when
subtracting the mean. We can compute explicitly $S(f)(w)$ using
\eqref{essa2} and the remark that the sum in $n$ and $k$ may be
reversed because the sum in $n$ is finite. We obtain
$$
S(f)(w)= - \sum_{k=1}^\infty \frac{\Phi(w)^k}{k}\, \frac{1}{2 \pi i}
\int_{|\tau|=1} \overline{\Phi(\tau)} \Phi'(\tau)
\frac{1}{\Phi(\tau)^k}\,d\tau.
$$
The only term that survives is that corresponding to the index $k=2$
and the result of the integral is $\xi.$ Thus
$$
S(f)(w)= - \frac{\xi}{2} (w+ \xi \overline{w})^2
$$
and
$$
2 \operatorname{Re}S(f)(w) = -\left(\frac{\xi}{2}(1+\xi^2)(w^2 +
\overline{w}^2)) + 2 \xi^2 \right).
$$
Again the constant term will disappear when subtracting the mean and
we conclude \mbox{that  $F(\lambda,f) \in Y.$}

A final remark is that, strictly speaking, the conclusion of
Crandall-Rabinowitz's Theorem is that for some little interval of
$\xi$ centered at $0$ the associated ellipse rotates. But, of
course, that any ellipse satisfies Burbea's equation
\eqref{rvortex5} can be proved directly.  It is interesting to
notice that, in this example, $\Phi_\xi(z)$ is analytic on a
neighborhood of $\{z: |z|\ge 1 \}$ for each $\xi \in (-1,1)$, and
real analytic in $\xi$ for each $z \in \T.$ We do not know how
general this fact is.

\section{Boundary smoothness of rotating vortex patches}

In this section we prove our main result, namely that if the
bifurcated patch is close enough to the circle where bifurcation
takes place, then the boundary of the patch is of class $C^\infty.$
Before stating the result more formally we remind the reader of the
big picture. We called $V$ the set of functions in the unit ball of
$C^{1+\alpha}(\T)$ with real Fourier coefficients living only at
negative frequencies. Each $f \in V$ determines a conformal mapping
$\Phi(z)=z+f(z)$ of the complement of the closed unit disc
$\overline{\Delta}$ into some domain containing the point at
$\infty.$ The boundary of the simply connected domain $D=\C
\setminus \Phi(\C\setminus \Delta)$ is the Jordan curve $\Phi(\T)$
and so, since $\Phi \in C^{1+\alpha}(\T),$ the boundary of $D$ is a
Jordan curve of class $C^{1+\alpha}.$ Burbea's existence Theorem
asserts that for each integer $m \geq 2$ there exists a small
positive number $a$ and a continuous curve $f(\xi),\, -a < \xi <a,$
taking values in $V$ such that the simply connected domain $D_\xi$
associated with $f(\xi)$ is an $m$-fold rotating vortex patch. Since
$f(0)=0,$ $D_0$ is the open unit disc and one should think that
$D_\xi$ is a domain close to the disc, for small values of $\xi \in
(-a,a)$, in the topology determined by $C^{1+\alpha}(\T)$ . We claim
that if $D_\xi$ is close enough to the disc in the topology given by
$C^{1}(\T)$ then the boundary of $D_\xi$ is of class $C^\infty.$
Later on we will show that if $D_\xi$ is close enough to the disc in
the topology given by $C^{2}(\T)$ then $D_\xi$ is also convex.

\begin{teor}\label{teor12}
For each integer $m \ge 3,$ there exists a small positive
$\epsilon_0 = \epsilon_0(m)$ such that if $f \in V$ defines an
$m$-fold V-state $D$ and $\|f\|_{C^1(\T)}< \epsilon_0$, then $D$ has
boundary of class $C^\infty.$
\end{teor}

\begin{proof}[\it Outline of the proof]

Before plunging into the details we present a sketch of the proof of
Theorem \ref{teor12}. Burbea's equation for $V$-states is
$F(\lambda,f)(w) = 0, \; w\in\T,$ with
$$
F(\lambda,f)(w) = \lambda |w+f(w)|^{2} +2\operatorname{Re} S(f)(w) -
m(\text{Id}+f,\lambda), \quad w\in\T,
$$
where $S(f)$ is given by \eqref{ops1} and $m(\text{Id}+f,\lambda)$
by \eqref{mitjana}. Recall that the reason to subtract
$m(\text{Id}+f,\lambda)$ is that the integral of $F(\lambda,f)$ over
$\T$ be zero.

In the previous section we have used bifurcation theory to prove
that, given an integer $m \geq 2$ and $0<\alpha<1,$ there exists a
curve of $V$-states passing through $(1/m,0)$ and taking values in a
little neighborhood of  $(1/m,0)$ in $(0,\infty)\times V.$ Since $V$
is contained in $C^{1+\alpha}(\T)$ we conclude that the non-trivial
$V$-states we have found have boundary of class $C^{1+\alpha}(\T).$
The same approach can be adapted with slight modifications to the
space $C^{n+\alpha}(\T)$, for each positive integer $n$ and each
$0<\alpha<1$. This provides curves of solutions with boundaries of
class $C^{n+\alpha}.$ However the neighborhood of $(1/m,0)$
containing the curve of solutions decreases as $n$ increases and so
the $C^\infty$ regularity of the  boundary cannot be reached by
using Crandall-Rabinowitz as a black box. Since Crandall-Rabinowitz
depends essentially on the implicit function theorem, we thought of
resorting to Nash-Moser implicit function theorem for
$C^\infty(\T).$ Unfortunately we were not able to implement this
idea. We realized later that a simpler method works. The idea is to
differentiate the equation $F(\lambda,f)(w)=0$ with respect to $w$
and carefully study the resulting equation. We find a surprising
smoothing effect for the unit tangent field to the curve $\Phi(\T)$
which induces in turn a global smoothing effect for the conformal
mapping $\Phi.$
 To be more precise, we
compute $\frac{dF(\lambda,f)}{dw},$ which yields a formula for the
quotient
$$
q(w) = \frac{\overline{\Phi'(w)}}{\Phi'(w)}, \quad w\in \mathbb{T},
$$
 of the form
\begin{equation}\label{eqq}
q(w)=w^2\frac{(1-\lambda)\overline{\Phi(w)}+{I}_1(w)}{(1-\lambda){\Phi(w)}+\overline{{I}_1(w)}},\quad
w\in \mathbb{T}.
\end{equation}
where $I_1$ is the integral
\begin{equation}\label{I1}
I_1(w)=\frac{1}{2\pi
i}\int_{\mathbb{T}}\frac{\overline{\Phi(\tau)-\Phi(w)}}{\Phi(\tau)-\Phi(w)}\Phi^\prime(\tau)d\tau.
\end{equation}
As we know, a priori $\Phi \in C^{1+\alpha}(\T)$ and thus $q$ is
only in $C^\alpha(\T).$ But \eqref{eqq} suggests that $q$ might be
of class $C^{1+\alpha}(\T)$ provided $I_1(w)$ is. It is not
difficult to compute the derivative of $I_1$ and check that it is in
$C^\beta(\T), \; 0 < \beta < \alpha.$ Thus we get that $q \in
C^{1+\beta}(\T), \;  0 < \beta < \alpha,$ if the denominator in
\eqref{eqq} does not vanish in $\T.$ This is guaranteed by the
smallness condition $\|f\|_{C^1(\T)} <\epsilon_0.$ Now the
smoothness of $q$ is the same as that of the unit tangent vector to
the curve $\Phi(\T)$ and classical results on the smoothness of
conformal mappings yield that $\Phi \in C^{2+\beta}(\T), \;  0 <
\beta < \alpha.$  One can then iterate the argument and show that
under the assumption $\Phi \in C^{2+\beta}(\T), \;  0 < \beta <
\alpha$ the function $I_1$ can be differentiated twice with respect
to $w$ and the second derivative is in $C^\beta(\T), \; 0 < \beta <
\alpha.$ Thus $\Phi \in C^{3+\beta}(\T), \; 0 < \beta < \alpha.$
Since the iteration can be performed any number of times we conclude
that $\Phi \in C^\infty(\T).$ We begin now with the details of the
proof.
\end{proof}

\begin{proof}[Proof of Theorem \ref{teor12}]
To get \eqref{eqq} we differentiate the equation $F(\lambda,f)(w)=0$
with respect to $w$ and use \eqref{diff1}. We get
$$
\lambda\Big(\Phi^\prime(w)\overline{\Phi(w)}-\frac{1}{w^2}\overline{\Phi^\prime(w)}{\Phi(w)}\Big)+\frac{d
S(f)}{dw}(w)-\frac{1}{w^2}\overline{\frac{d S(f)}{dw}}(w)=0.
$$
Recall that the parameter $\lambda$ is taken in the interval
$]0,1[$. According to  \eqref{dessa} and \eqref{identit123},
\begin{eqnarray*}
\frac{d S(f)}{dw}(w)&=&-\Phi^{\prime}(w)\bigg(\overline{\Phi(w)}+ \frac{1}{2 \pi i}\int_{\T}\frac{\overline{\Phi(\tau)-\Phi(w)}}{\Phi(\tau)-\Phi(w)}\Phi^\prime(\tau)d\tau\bigg)\\
&:=&-\Phi^{\prime}(w)\Big(\overline{\Phi(w)}+{I}_1(w)\Big),\quad
w\in\T.
\end{eqnarray*}
Putting together the two preceding  identities and setting
$q(w):=\frac{\overline{\Phi^\prime(w)}}{\Phi^\prime(w)}$  we obtain
\eqref{eqq}.
 Let us show that the
denominator in \eqref{eqq} does not vanish on $\T$ if
$\|f\|_{C^1(\T)}$ is small enough. \\
Since $\Phi(w)=w + f(w), \; w\in\T, $ the denominator in \eqref{eqq}
is
\begin{equation}\label{denominator}
D(w)=(1-\lambda)\big(w+f(w)\big)+\overline{I_1(w)},\quad w\in
\mathbb{T}.
\end{equation}
Now
\begin{eqnarray}\label{I12}
\nonumber I_1(w)&=&\frac{1}{2\pi
i}\int_{\mathbb{T}}\frac{\overline{\tau-w}}
{\Phi(\tau)-\Phi(w)}\Phi^\prime(\tau)d\tau+\frac{1}{2\pi i}\int_{\mathbb{T}}\frac{\overline{f(\tau)-f(w)}}{\Phi(\tau)-\Phi(w)}\Phi^\prime(\tau)d\tau\\
\nonumber&=& - \frac{1}{w}\frac{1}{2\pi
i}\int_{\mathbb{T}}\frac{{\tau-w}}{\Phi(\tau)-\Phi(w)}\Phi^\prime(\tau)\frac{d\tau}{\tau}+J_1(w).
\end{eqnarray}
where $J_1$ is a notation for the second term and we have used the
identity
\begin{equation}\label{tao}
\overline{\tau-w}=-\frac{\tau-w}{\tau w},\quad \tau, \, w\in\T.
\end{equation}
 On the other hand, by Lebesgue dominated convergence Theorem we
 have
\begin{equation*}
\frac{1}{2\pi
i}\int_{\mathbb{T}}\frac{{\tau-w}}{\Phi(\tau)-\Phi(w)}\Phi^\prime(\tau)\frac{d\tau}{\tau}=\lim_{\varepsilon\to
0}\frac{1}{2\pi
i}\int_{|\tau|=1+\varepsilon}\frac{{\tau-w}}{\Phi(\tau)-\Phi(w)}\Phi^\prime(\tau)\frac{d\tau}{\tau}.
\end{equation*}
The integral on the circle of radius $1+\varepsilon$ is $1$ for each
positive $\varepsilon,$ because the integrand is an analytic
function of $\tau$ in the exterior of the unit disc whose first term
in the expansion at $\infty$ is $1/\tau. $ Hence
$$
 I_1(w)=-\overline{w}+J_1(w), \quad  w\in \T.
$$
Plugging this identity into \eqref{denominator}
$$
D(w)=-\lambda w+(1-\lambda) f(w)+ \overline{J_1(w)}, \quad w\in\T.
$$
The estimate of the $L^\infty$ norm of $J_1$ can be easily performed
as follows
\begin{eqnarray*}
\|J_1\|_{L^\infty(\T)}&\le& \frac{\|f^\prime\|_{L^\infty(\T)}}{\inf_{\tau\neq w}\frac{|\Phi(\tau)-\Phi(w)|}{|\tau-w|}}\|\Phi^\prime\|_{L^\infty(\T)}\\
&\le&\frac{\|f^\prime\|_{L^\infty(\T)}}{1-\|f^\prime\|_{L^\infty(\T)}}\big(1+\|f^\prime\|_{L^\infty}\big).
\end{eqnarray*}
Therefore
\begin{eqnarray*}
|D(w)|&\geq& \lambda-(1-\lambda)\|f\|_{L^\infty(\T)}-\|\tilde{I}_1\|_{L^\infty(\T)}\\
&\ge&
\lambda-\|f\|_{\infty}-\|f^\prime\|_{\infty}\frac{1+\|f^\prime\|_{\infty}}{1-\|f^\prime\|_{\infty}}\\
&\geq& \lambda-2\frac{\|f\|_{C^1(\T)}}{1-\|f\|_{C^1(\T)}}\\
&\geq&\frac{1}{2}\lambda,
\end{eqnarray*}
where the last inequality holds provided $\|f\|_{C^1(\T)}
\le\frac{\lambda}{4+\lambda}.$

Let us  now prove  that the function $w\in\T\mapsto I_1(w)$ is more
regular than one would expect. Indeed, it belongs to the space $
C^{1+\beta}(\T) $ for any $\beta$ satisfying $0 < \beta < \alpha.$
Since the quotient $w\mapsto
\frac{\overline{\Phi(\tau)-\Phi(w)}}{\Phi(\tau)-\Phi(w)} $ extends
continuously to the diagonal of $\T$, in taking derivatives inside
the integral defining $I_1(w)$ no ``boundary terms" will appear.
Then it follows from \eqref{diff1} and \eqref{identit123} that
\begin{eqnarray}\label{deI1}
\nonumber\frac{dI_1}{dw}(w)&=&\frac{1}{w^2}\overline{\Phi^\prime(w)}\,\,\text{p.v.}\,\frac{1}{2
\pi i}\int_{\T}\frac{\Phi^\prime(\tau)}{\Phi(\tau)-\Phi(w)}d\tau \\
&+&  \Phi'(w)\,\, \text{p.v.}\,\frac{1}{2\pi
i}\int_{\T}\frac{\overline{\Phi(\tau)-\Phi(w)}}{(\Phi(\tau)-\Phi(w))^2}\Phi^\prime(\tau)d\tau\\
&=& \frac{\overline{\Phi^\prime(w)}}{2w^2}+\Phi^\prime(w) \;
\text{p.v.\ }\,\frac{1}{2\pi
i}\int_{\T}\frac{\overline{\Phi(\tau)-\Phi(w)}}{(\Phi(\tau)-\Phi(w))^2}\Phi^\prime(\tau)d\tau.
\end{eqnarray}
To obtain the appropriate H\"{o}lder estimate we are looking for it
is convenient to write the principal value integral above as the sum
of two terms by adding and subtracting
$\overline{\Phi'(w)}\overline{(\tau-w)} $ in the numerator of the
fraction.  We get
\begin{eqnarray}\label{I2}
\nonumber\text{p.v.\ }\,\frac{1}{2\pi
i}\int_{\T}\frac{\overline{\Phi(\tau)-\Phi(w)}}{(\Phi(\tau)-\Phi(w))^2}\Phi^\prime(\tau)d\tau
\nonumber &=&\frac{1}{2\pi
i}\int_{\T}\frac{\overline{\Phi(\tau)-\Phi(w)-\Phi^\prime(w)(\tau-w)}}{(\Phi(\tau)-\Phi(w))^2}
 \Phi^\prime(\tau)d\tau\\
 &+ &\overline{\Phi^\prime(w)}\; \text{p.v.\ }\,\frac{1}{2\pi i}\int_{\T}\frac{\overline{\tau-w}}{(\Phi(\tau)-\Phi(w))^2}\Phi^\prime(\tau)d\tau
\end{eqnarray}
Call $I_2(w)$ the first term in the right-hand side above, namely,
$$
I_2(w):=\frac{1}{2\pi
i}\int_{\T}\frac{\overline{\Phi(\tau)-\Phi(w)-\Phi^\prime(w)(\tau-w)}}{(\Phi(\tau)-\Phi(w))^2}
 \Phi^\prime(\tau)d\tau,\quad w\in\T.
$$ Let us compute the
principal value integral in the second term of \eqref{I2}. By
\eqref{tao}
\begin{equation}\label{taoepsilon}
\text{p.v.\ }\,\frac{1}{2\pi
i}\int_{\T}\frac{\overline{\tau-w}}{(\Phi(\tau)-\Phi(w))^2}\Phi^\prime(\tau)d\tau=-\frac1w\,
\text{p.v.\ }\,\frac{1}{2\pi
i}\int_{\T}\Big(\frac{\tau-w}{\Phi(\tau)-\Phi(w)}\Big)^2\
\frac{\Phi^\prime(\tau)}{\tau}\frac{d\tau}{\tau-w}\cdot
\end{equation}
We compute the principal value integral above by the method used in
dealing with \eqref{tione}.  Denote by $\gamma_\epsilon,\;
\epsilon>0,$ the arc which is the intersection of the circle
centered at $w$ of radius $\epsilon$ and the complement of the open
unit disc, with the counter-clockwise orientation. Let $\T_\epsilon$
the closed Jordan curve consisting of the arc $\gamma_\epsilon$
followed by the part of the unit circle at distance from $w$ not
less than $\epsilon,$  traversed counterclockwise. Then the
principal value in \eqref{taoepsilon} is
\begin{eqnarray*}
\nonumber\text{p.v.\ }\,\frac{1}{2\pi
i}\int_{\T}\Big(\frac{\tau-w}{\Phi(\tau)-\Phi(w)}\Big)^2\
\frac{\Phi^\prime(\tau)}{\tau}\frac{d\tau}{\tau-w} &=&
  \lim_{\varepsilon\to 0}\,\frac{1}{2\pi
i}\int_{\mathbb{T}_\varepsilon}\Big(\frac{\tau-w}
{\Phi(\tau)-\Phi(w)}\Big)^2\,\frac{\Phi^\prime(\tau)}{\tau}\frac{d\tau}{\tau-w}\\
& -& \lim_{\varepsilon\to 0}\, \frac{1}{2\pi i}
\int_{\gamma_\varepsilon}\Big(\frac{\tau-w}{\Phi(\tau)-\Phi(w)}\Big)^2\,\frac{\Phi^\prime(\tau)}{\tau}\frac{d\tau}{\tau-w}.
\end{eqnarray*}
The integral on $\T_\epsilon$ is zero because the integrand is
analytic in the exterior of the unit disc and has a double zero at
$\infty.$ The limit of the integral on $\gamma_\epsilon$ is given by
$$
 \lim_{\varepsilon\to 0}\, \frac{1}{2\pi i}
\int_{\gamma_\varepsilon}\Big(\frac{\tau-w}{\Phi(\tau)-\Phi(w)}\Big)^2\,\frac{\Phi^\prime(\tau)}{\tau}\frac{d\tau}{\tau-w}=\frac{1}{2w\Phi^\prime(w)}\cdot
$$
Therefore
\begin{equation}\label{deItilde}
\frac{dI_1(w)}{dw}=\Phi^\prime(w)
I_2(w)+\frac{\overline{\Phi^\prime(w)}}{w^2},\quad  w\in\T.
\end{equation}
 Lemma \ref{nuclitaylor} below applied  for $n=2$ yields
$I_2 \in  C^\beta(\T),$ for each $\beta$ satisfying $  0 < \beta <
\alpha$ and hence we conclude that $w\mapsto \frac{dI_1}{dw}(w)$
belongs to the space $ C^{\beta}(\T),$ $ \;  0 < \beta < \alpha,$
that is, that $I_1 \in C^{1+\beta}(\T),$ $0 < \beta < \alpha.$ Thus
from the expression \eqref{eqq}, we find that
 \begin{equation}\label{smooth00}
 q \in C^{1+\beta}(\T),\quad  0 < \beta < \alpha.
 \end{equation}

Before dealing with Lemma \ref{nuclitaylor} we discuss how the
smoothness of $q$ translates into the same type of smoothness of
$\Phi'.$ More precisely, we will prove the following: for each
positive integer $n$ and $0<\beta<1,$
\begin{equation}\label{smooth1}
q \in C^{n+\beta}(\T;\T)\Longrightarrow \Phi \in
C^{n+1+\beta}(\T;\mathbb{C}).
\end{equation}
For this purpose, we first relate the regularity of the map
$w\mapsto q(w)$ to the smoothness of the Jordan curve $\Phi(\T)$
 and then we use the Kellogg-Warschawski Theorem \cite{WS} to get a global regularity result for the conformal map $\Phi.$

Using the conformal parametrization $\theta\in\R\mapsto
\Phi(e^{i\theta})$, we get easily the following formula for the unit
tangent vector $\vec{\tau}(\theta)$ to the curve $\Phi(\T)$ at the
point $\Phi(e^{i\theta}),$
$$
\vec{\tau}(\theta)=\frac{\frac{d}{d\theta}\Phi(e^{i\theta})}{|\frac{d}{d\theta}\Phi(e^{i\theta})|}=iw\frac{\Phi^\prime(w)}{|\Phi^\prime(w)|},\quad
w=e^{i\theta}\cdot
$$
Consequently
\begin{equation}\label{regul1}
[\vec\tau(\theta)]^2=-w^2\overline{q(w)},\quad w=e^{i\theta}.
\end{equation}
Since $\Phi$ belongs to $C^{1+\alpha},$ the map $\theta\mapsto
\vec{\tau}(\theta)$ must be in $C^\alpha(\mathbb{R}; \T)\subset
C(\mathbb{R};\T)$ and by the lifting theorem there exists a
continuous function $\phi:{\mathbb{R}}\to\mathbb{R},$ such that
$$
\vec\tau(\theta)=e^{i\phi(\theta)},\,\theta\in\mathbb{R}.
$$
Recall that we have established that $q\in C^{1+\beta}(\T),
0<\beta<\alpha$  and so $[\vec{\tau}]^2$ remains in  the same space.
Since the argument function $\phi$ can be recovered by the formula
$$
\phi(\theta)=\phi(0)-\frac12i\int_0^\theta\frac{\sigma^\prime(t)}{\sigma(t)}dt,\quad\quad
\hbox{with}\quad \quad \sigma(t):=[\vec{\tau}(t)]^2,
$$
 $\phi$ is in $C^{1+\beta}(\mathbb{R})$ and consequently
$\vec{\tau} \in C^{1+\beta}(\R).$ More generally, the preceding
formula for $\phi$ shows that, for each non-negative integer $n$ and
${\beta}\in]0,1[,$
$$
q\in C^{n+\beta}(\T;\T)\,\Longrightarrow\,\vec{\tau}\in
C^{n+\beta}(\mathbb{R};\T).
$$
Now we will use the Kellogg-Warschawski theorem \cite{WS}, which can
be also found in \cite[Theorem 3.6]{P}.  It asserts that if the
boundary $\Phi(\T)$ is a Jordan curve of class $C^{n+1+\beta},$ with
$n$ a non-negative integer and $ 0<\beta<1$, then the conformal map
$\Phi: \C\backslash\overline{\Delta}\to \C$ has a continuous
extension to $\C\backslash{\Delta}$  which is of class
$C^{n+1+\beta}.$  In other words,
$$
\vec{\tau}\in C^{n+\beta}(\mathbb{R};\T)\Longrightarrow \Phi \in
C^{n+1+\beta}(\T;\mathbb{C}).
$$
Combining \eqref{smooth00} and \eqref{smooth1} we obtain
$$
\Phi\in C^{2+\beta}(\T).
$$
%
We are now ready to iterate the preceding argument. Assume that
$\Phi \in C^{n-1+\beta}(\T),$ \newline $0 < \beta < 1,$ for some $n
\ge 3.$ We are going to show that,
\begin{equation}\label{recursi}
\Phi \in C^{n+\gamma}(\T),\quad  0 < \gamma < \beta.
\end{equation}
This will complete the proof that $\Phi \in C^{\infty}(\T).$
\\We need the following general
lemma. 
For $\Phi \in C^n(\T),$ let
$$
P_n(\Phi)(\tau,w)= \sum_{j=0}^n \frac{\Phi^{(j)}(w)}{j !} \,(\tau
-w)^j
$$
be the Taylor polynomial of degree $n$ of $\Phi,$ around the point
$w$, evaluated at the point  $\tau.$
\begin{lemma}\label{nuclitaylor}
Assume that $\Phi \in C^{n-1+\alpha}(\T),$\, for \,$n\geq 2,\,
0<\alpha<1.$ Let $T_n$ be the operator
\begin{equation}\label{teena}
T_ng(w)=\frac{1}{2\pi i}\int_{\T}K_n(w,\tau) g(\tau)d\tau, \quad
w\in\T,\quad g\in L^\infty(\mathbb{T}),
\end{equation}
with kernel
\begin{equation*}
K_n(w,\tau)=\frac{\overline{\Phi(\tau)-P_{n-1}(\Phi)(\tau,w)}}{(\Phi(\tau)-\Phi(w))^n}.
\end{equation*}
Then, for any $\beta$ satisfying $0 < \beta < \alpha,$
\begin{equation*}
\|T_ng\|_{C^\beta(\mathbb{T})}\le C\|g\|_{L^\infty(\T)}.
\end{equation*}
\end{lemma}
\begin{proof}
The lemma is easily proven by standard methods, as in \cite[p.419
]{MOV}, once one knows that $K_n$ satisfies
\begin{equation}\label{kena1}
|K_n(w,\tau)\le C|\tau-w|^{\alpha-1},\quad \tau,w\in
\mathbb{T},\tau\neq w
\end{equation}
and
\begin{equation}\label{kena2}
|K_n(w_1,\tau)-K_n(w_2,\tau)|\le
C\frac{|w_1-w_2|^\alpha}{|\tau-w_1|}, \quad w_1,w_2 \in
\mathbb{T},\; |\tau-w_1|\geq 2|w_1-w_2|.
\end{equation}
It is obvious that \eqref{kena1} holds by Taylor's formula. For
\eqref{kena2} we write
%
\begin{eqnarray*}
\nonumber|K_n(w_1,\tau)-K_n(w_2,\tau)|&\le&\bigg|\frac{\overline{P_{n-1}(\Phi)(\tau,w_1)}-\overline{P_{n-1}(\Phi)(\tau,w_2)}}{(\Phi(\tau)-\Phi(w_1))^n}
\bigg|\\
\nonumber& +& \Big|\Phi(\tau)-P_{n-1}(\Phi)(\tau,w_2)\Big|\,\Big|\frac{1}{(\Phi(\tau)-\Phi(w_1))^n}-\frac{1}{(\Phi(\tau)-\Phi(w_2))^n}\Big|\\
&:=& \hbox{I}+\hbox{II}.
\end{eqnarray*}
The term $\hbox{II}$ can easily be controlled via a gradient
estimate by
\begin{eqnarray*}
|\hbox{II}|&\le&
 C|\tau-w_2|^{n-1+\alpha} \frac{|w_1-w_2|}{|\tau-w_1|^{n+1}}\\
&\le& C\frac{|w_1-w_2|^\alpha}{|\tau-w_1|}.
\end{eqnarray*}
We have used in the  last inequality the equivalence $
\frac12|\tau-w_1|\le|\tau-w_2|\le\frac32|\tau-w_1|$. The \mbox{term
$I$} is estimated by observing that there is an elementary formula
for the difference of two Taylor's polynomials around different
points $w_1$ and $w_2,$  namely,
\begin{equation}\label{diftaylor}
P_{n-1}(\Phi)(\tau,w_2)-P_{n-1}(\Phi)(\tau,w_1)=
\sum_{j=0}^{n-1}\Big(\Phi^{(j)}(w_2)-P_{n-1-j}
(\Phi^{(j)})(w_2,w_1)\Big)\frac{(\tau-w_2)^j} {j!}\cdot
\end{equation}
This follows easily  from the  identity,
\begin{eqnarray*}
P_{n-1}(\Phi)(\tau,w_1)&=&\sum_{j=0}^{n-1}\frac{(\tau-w_2)^j}{j!}\{\partial_\tau^{(j)}[P_{n-1}(\Phi)]\}(w_2,w_1)\\
&=&\sum_{j=0}^{n-1}\frac{(\tau-w_2)^j}{j!}\sum_{k=j}^{n-1}\frac{(w_2-w_1)^{k-j}}{(k-j)!}\Phi^{(k)}(w_1)\\
&=&\sum_{j=0}^{n-1}\frac{(\tau-w_2)^j}{j!}P_{n-1-j}(\Phi^{(j)})(w_2,w_1).
\end{eqnarray*}
Since $\Phi^{(j)}$ belongs to $C^{n-1-j+\alpha}(\mathbb{T})$ then
$$
\big| \Phi^{(j)}(w_2)-P_{n-1-j}(\Phi^{(j)})(w_2,w_1) \big|\le C|w_2-w_1|^{n-1-j+\alpha}.
$$
Combining this estimate with formula \eqref{diftaylor} yields

\begin{eqnarray*}
\quad
|\hbox{I}|&\le& C\sum_{j=0}^{n-1}|w_2-w_1|^{n-1-j+\alpha}\frac{|\tau-w_2|^j}{j!}\frac{1}{|\tau-w_1|^n}\\
&\le&CÊ\frac{|w_1-w_2|^\alpha}{|\tau-w_1|}\cdot
\end{eqnarray*}
\end{proof}

In view of \eqref{eqq} the only task left is to show that $I_1 \in
C^{n-1+\gamma}(\T),\, 0 < \gamma<\beta,$ provided $\Phi \in
C^{n-1+\beta}(\T).$ Indeed, this will lead to $q\in
C^{n-1+\gamma}(\T)$ and, according to the discussion above on the
Kellog-Warschawski Theorem, we conclude that $\Phi\in
C^{n+\gamma}(\T).$ Now, in order to prove that $I_1 \in
C^{n-1+\gamma}(\T)$ we need to establish a recursive formula for the
higher order derivatives of $I_1$, which is the goal of the next
lemma.

\begin{lemma}\label{deI}
Let $\Phi \in C^{n+\beta}(\T), 0 < \beta < 1, n \ge 2,$ and set
\begin{equation*}
I_n(w):=\frac{1}{2\pi
i}\int_{\T}\frac{\overline{\Phi(\tau)-P_{n-1}(\Phi)(\tau,w)}}
{\big(\Phi(\tau)-\Phi(w)\big)^n}\Phi^\prime(\tau)d\tau,\quad w\in\T.
\end{equation*}
Then
\begin{equation}\label{deIn}
\frac{dI_n(w)}{dw}=n\Phi^\prime(w)I_{n+1}(w),\,w\in\T.
\end{equation}
\end{lemma}

\begin{remark}
Notice that formula \eqref{deI1}  for the derivative of $I_1$ falls
out of the scope of \eqref{deIn}. \vspace{0.2 cm} Indeed, it is a
fortunate fact that a compact formula as \eqref{deIn} can be found.
\end{remark}

The proof of the preceding lemma depends on the following
calculation.

\begin{sublemma}\label{sublemma}
Let $\Phi \in C^{1}(\T) $ and $n \ge 1.$ Then
\begin{equation}\label{deInpv}
\text{p.v.\ }\frac{1}{2\pi
i}\int_{\T}\frac{(\tau-w)^n}{\big(\Phi(\tau)-\Phi(w)\big)^{n+1}}\Phi^\prime(\tau)
\,\frac{d\tau}{\tau^n}= -\frac{1}{2\, \Phi^\prime(w)^nw^n},\quad
w\in \T.
\end{equation}
\end{sublemma}
\begin{proof}
Consider again the closed Jordan curve $\T_\epsilon$ and the arc
$\gamma_\epsilon$ used to deal with \eqref{tione} and
\eqref{taoepsilon}. The principal value integral in the statement of
the sublemma is the limit, as $\epsilon$ tends to $0,$ of the sum of
two terms. The first is the integral over $\T_\epsilon$ of the
integrand in \eqref{deInpv},  which is zero because the integrand is
an analytic function of $\tau$ in the exterior of the unit disc with
a zero at $\infty$ of order at least $2.$  The second term is minus
the limit as $\epsilon$ tends to $0$ of the integral of the same
expression over the arc $\gamma_\epsilon.$ Since
$$
\frac{1}{2 \pi i} \lim_{\epsilon \rightarrow 0}
\int_{\gamma_\epsilon} \frac{\Phi'(\tau)}{\Phi(\tau)-\Phi(w)}\,d\tau
= \frac{1}{2 \pi i} \lim_{\epsilon \rightarrow 0}
\Delta_{\gamma_\epsilon} \log \Phi = \frac{1}{2}
$$
it is clear that the limit of the second term is $-1 / (2 \,
\Phi'(w)^n
 w^n).$ Here $\Delta_{\gamma_\epsilon} \log \Phi$ stands for
 the variation of $\log \Phi$ on the arc $\gamma_\epsilon.$
\end{proof}

\begin{proof}[Proof of Lemma \ref{deI}]
Since $\Phi \in C^{n}(\T),$ the fraction in the integrand in $I_n$
extends continuously to the diagonal of $\T$ taking the value
$(-1)^n \overline{\Phi^{(n)}(w)}\, \overline{w}^{2n}/ (n!
\,\Phi'(w)^n) .$ We can then take derivatives inside the integral
and the boundary terms arising in the integration by parts are zero.
Thus
\begin{eqnarray*}
\frac{dI_n}{dw}(w)&=&n\Phi^\prime(w)\,\text{p.v.\ } \frac{1}{2\pi
i}\int_{\T}\frac{\overline{\Phi(\tau)-P_{n-1}(\Phi)(\tau,w)}}{\big(\Phi(\tau)-\Phi(w)\big)^{n+1}}\Phi^\prime(\tau){d\tau}\\
&+&\frac{1}{n-1 !}\frac{\overline{\Phi^{(n)}(w)}}{w^2}\,\hbox{p.v.\
} \frac{1}{2\pi
i}\int_{\T}\frac{(\overline{\tau-w})^{n-1}}{\big(\Phi(\tau)-\Phi(w)\big)^{n}}\Phi^\prime(\tau){d\tau}\\
&:=&T_1(w)+T_2(w).
\end{eqnarray*}
The next move consists in adding and subtracting
$\overline{\Phi^{(n)}(w) (\tau-w)^n} /n!$ to the numerator of the
fraction in the integrand of $T_1(w).$ The result is
\begin{eqnarray*}
T_1(w)&=&n\Phi^\prime(w)\,\text{p.v.\ } \frac{1}{2\pi
i}\int_{\T}\frac{\overline{\Phi(\tau)-P_{n}(\Phi)(\tau,w)}}{\big(\Phi(\tau)-\Phi(w)\big)^{n+1}}\Phi^\prime(\tau){d\tau}\\
&+&
\Phi^\prime(w)\frac{{\overline{\Phi^{(n)}(w)}}}{n-1!}\,\,\hbox{p.v.}\,
\frac{1}{2\pi
i}\int_{\T}\frac{(\overline{\tau-w})^{n}}{\big(\Phi(\tau)-\Phi(w)\big)^{n+1}}\Phi^\prime(\tau){d\tau}\\
&=&n\Phi^\prime(w)
I_{n+1}(w)+\Phi^\prime(w)\frac{{\overline{\Phi^{(n)}(w)}}}{n-1!}\,\,\hbox{p.v.}\,
\frac{1}{2\pi
i}\int_{\T}\frac{(\overline{\tau-w})^{n}}{\big(\Phi(\tau)-\Phi(w)\big)^{n+1}}\Phi^\prime(\tau){d\tau}\\
&:=&n\Phi^\prime(w) I_{n+1}(w)+T_3(w)
\end{eqnarray*}

We claim that $T_2(w)+T_3(w) =0,$ which ends the proof of the lemma.
The only difficulty is to compute the principal value integrals,
which are the same except for a shift in the exponents of the
integrand. For instance, by \eqref{tao} and the sublemma we see that
the principal value integral in the term $T_3(w)$ is
\begin{eqnarray*}
\hbox{p.v.}\, \frac{1}{2\pi
i}\int_{\T}\frac{(\overline{\tau-w})^{n}}{\big(\Phi(\tau)-\Phi(w)\big)^{n+1}}\Phi^\prime(\tau){d\tau}&=&\frac{(-1)^n}{w^n}
\hbox{p.v.}\, \frac{1}{2\pi
i}\int_{\T}\frac{({\tau-w})^{n}}{\big(\Phi(\tau)-\Phi(w)\big)^{n+1}}
\Phi^\prime(\tau)\frac{d\tau}{\tau^n}\\
&=&\frac{(-1)^{n+1}}{2[\Phi^\prime(w)]^nw^{2n}}.
\end{eqnarray*}
This completes the proof of the Lemma \ref{deI}.
\end{proof}

We can now finish the proof of Theorem \ref{teor12}. Recall that our
assumption is $\Phi \in C^{n-1+\beta}(\T),$  where $ 0 < \beta <1$
and  $n \ge 3,$ and  we want to conclude that $I_1 \in
C^{n-1+\gamma}(\T),$ $0 < \gamma < \beta.$ We have already seen that
this will give  $\Phi \in C^{n+\gamma}(\T),\; 0 < \gamma < \beta. $
  We apply Lemma \ref{nuclitaylor} and we get
that $I_n= T_n(\Phi')$ belongs to  $C^\gamma(\T), \; 0 < \gamma <
\beta.$ Formula \eqref{deIn} in Lemma \ref{deI} readily yields
$I_{n-1} \in C^{1+\gamma}(\T).$ Iterating the use of \eqref{deIn} we
obtain $I_{2} \in C^{n-2+\gamma}(\T),$ and so, finally, by
\eqref{deI1}, $I_{1} \in C^{n-1+\gamma}(\T).$
\end{proof}

We end the paper with the following remark.

\begin{co}\label{convex}
For each integer $m \ge 3,$ there exists a small positive
$\epsilon_0 = \epsilon_0(m)$ such that if $f \in V$ defines an
$m$-fold V-state $D$,  $\|f\|_{C^1(\T)}< \epsilon_0$ and
$\|f\|_{C^2(\T)}< 1/2$, then $D$ is convex.
\end{co}

\begin{proof}
As it is well-known, if $D$ is a Jordan domain bounded by a smooth
Jordan curve of class $C^2$, then $D$ is convex if and only if the
curvature of the boundary curve does not change sign on the curve.
By Theorem \ref{teor12} we know that the boundary of our $V$-state
$D$ is of class $C^\infty.$ To compute the curvature at boundary
points we resort to the conformal parametrization
$\Phi(e^{i\theta}).$ The velocity vector (or tangent vector) and the
principal normal to the curve at the point $\Phi(w), w=e^{i\theta},$
are given by
$$
\vec{v}(\theta)=iw{\Phi^\prime(w)}\quad\quad \text{and} \quad\quad
\vec{n}(\theta)=-\frac{w \Phi^\prime(w)}{|\Phi^\prime(w)|}
$$
respectively. On the other hand, according to a well-known classical
formula for  the curvature $\kappa(\theta)$ at the point
$\Phi(e^{i\theta}),$ we have
$$
\kappa(\theta)=\frac{\operatorname{Re}\big({d_{\theta}^2\{\Phi(e^{i\theta})\}}\;\overline{\vec{n}(\theta)}\big)}{|\vec{v}(\theta)|^2}\cdot
$$
A  straightforward computation yields
$$
d_{\theta}^2\{\Phi(e^{i\theta})\}=-w\big(\Phi^\prime(w)+w\Phi^{\prime\prime}(w)\big).
$$
Thus the curvature is
$$
\kappa(\theta)=\frac{1}{|\Phi^\prime(w)|}\operatorname{Re}\Big(1+w\frac{\Phi^{\prime\prime}(w)}{\Phi^\prime(w)}\Big).
$$
Since $\Phi(w)=w+f(w)$
$$
1+w\frac{\Phi^{\prime\prime}(w)}{\Phi^\prime(w)}=1+ w
\frac{f^{\prime\prime}(w)}{1+f^\prime(w)}
$$
and so
$$
\operatorname{Re}\Big(1+w\frac{\Phi^{\prime\prime}(w)}{\Phi^\prime(w)}\Big)
\geq 1- \frac{|f^{\prime\prime}(w)|}{1-|f^\prime(w)|} \geq 1-
\frac{\|f\|_{C^2(\T)}}{1-\|f\|_{C^2(\T)}},
$$
which is non-negative if $\|f\|_{C^2(\T)} < 1/2.$
\end{proof}

The reader will find interesting information on conformal mappings
and convexity in Duren's book \cite{Du}.

\begin{gracies}
J. Mateu and J. Verdera are grateful to J.A. Carrillo for many enlightening conversations on fluid mechanics and to L. Vega for
suggesting that one should look at the special nature of the
equation to obtain our regularity result. They also acknowledge
generous support from the grants 2009SGR420 (Generalitat de
Catalunya) and MTM2010-15657 (Ministerio de Ciencia e
Innovaci\'{o}n).
\end{gracies}

\end{document}